\documentclass[11pt]{amsart}
\usepackage{amsmath}
\usepackage{amssymb,amscd}
\usepackage{epsfig}
\usepackage[curve]{xypic}
\usepackage{graphicx}
\usepackage{enumitem}
\usepackage{tikz-cd}

\usepackage{parskip}
\usepackage[colorlinks=true, urlcolor=blue,bookmarks=true,bookmarksopen=true, citecolor=blue]{hyperref}
\numberwithin{equation}{section}

\newtheorem{theorem}{Theorem}[section]

\newtheorem{claim}[theorem]{Claim}

\newtheorem{corollary}[theorem]{Corollary}

\newtheorem{lemma}[theorem]{Lemma}
\newtheorem{prop}[theorem]{Proposition}
\newtheorem{question}[theorem]{Question}

\theoremstyle{definition}
\newtheorem{defn}[theorem]{Definition}

\newtheorem{conj}[theorem]{Conjecture}

\newtheorem{example}[theorem]{Example}
\newtheorem{ques}[theorem]{Question}
\newtheorem{remark}[theorem]{Remark}
\newcommand{\st}{\mbox{\upshape st}\,}
\newcommand{\lk}{\mbox{\upshape lk}\,}
\newcommand{\Sp}{\mathbb{S}}
\newcommand{\RP}{\mathbb{RP}}
\newcommand{\D}{\Delta}

\def \begineq{\begin{equation}}
\def \endeq{\end{equation}}

\def \bb{\mathbb}

\def \RR{{\bb{R}}}

\def \ZZ{{\bb{Z}}}

\def \({\left(}
\def \){\right)}
\def \<{\langle}
\def \>{\rangle}
\def \bar{\overline}

\begin{document}

\title[On the equivariant triangulation of some small covers]{On the equivariant triangulation of some small covers}

\author[R. K. Gupta]{Raju Kumar Gupta}

\address{Department of Mathematics, Indian Institute of Technology Madras, Chennai-600036, India.}
\email{rajukrg3217@gmail.com}

\author[S. Sarkar]{Soumen Sarkar}

\address{Department of Mathematics, Indian Institute of Technology Madras, Chennai-600036, India.}

\email{soumen@iitm.ac.in}



\abstract In this paper, we study certain properties of $\mathbb{Z}_2^n$-equivariant triangulations of small covers.  We show that any $\mathbb{Z}_2^n$-equivariant triangulation of a small cover  naturally induces a triangulation of the orbit space. Then, we explicitly construct the minimal $\mathbb{Z}_2^3$-equivariant triangulation of $\mathbb{RP}^3$, which contains $11$ vertices and prove that this is the unique $\mathbb{Z}_2^3$-equivariant triangulation of $\mathbb{RP}^3$ with $11$ vertices. For a finite group $G$, we give a method for constructing some $G$-equivariant triangulations of connected sums of manifolds from their respective $G$-equivariant triangulations. In particular, we construct a $\mathbb{Z}_2^3$-equivariant triangulation of $\mathbb{RP}^3 \# \mathbb{RP}^3$ with $17$ vertices, which is the best known yet. This triangulation of $\mathbb{RP}^3 \# \mathbb{RP}^3$ provides another minimal $g$-vector improving one of the result of Lutz in \cite{LS}. Moreover, we prove that a $\ZZ_2^4$-equivariant triangulation of $\mathbb{RP}^4$ requires at least $18$ vertices. 

\endabstract

\maketitle

\section{Introduction}

The rich interplay between topology and combinatorics has led to many valuable discoveries, deepening our understanding of both fields. Triangulations of manifolds highlight such a connection by breaking complex spaces into simpler building blocks called simplices. This method supports both theoretical insights and computational applications. For decades, researchers have been interested in creating new triangulations and investigating the combinatorial features of topological spaces, see \cite{bjorner2000simplicial},\cite{kuhnel1983unique}.  The question of the size of a triangulation has gained practical importance. A triangulation $X$ of a topological space $M$ is said to be {\em minimal} if, for any other triangulation $Y$ of $M$, we have $f_0(X) \leq f_0(Y)$, where $f_0$ denotes the number of vertices.
It is worthwhile to investigate the number of vertices and the total number of faces in a triangulation, as well as to explicitly construct minimal triangulations. Constructing new triangulations (possibly with minimal vertices) for projective spaces has been one of the known research topics in this area. 

The minimal triangulation of $\mathbb{RP}^2$ is given by $6$ vertices, which is unique, \cite{KB}. A minimal triangulation of $\mathbb{RP}^3$ was constructed by Walkup with $11$ vertices, cf. \cite{Walkup}. In 1999, Lutz \cite{Lutz16vertex} constructed a $16$-vertex triangulation of $\mathbb{RP}^4$ and showed that it is a minimal triangulation of $\mathbb{RP}^4$. In \cite{balagopalan}, Balagopalan constructed 3 triangulation of $\mathbb{RP}^4$ in different ways, which are eventually the same as the Lutz triangulation. 

\begin{conj}\cite{Lutz16vertex}\label{RP4conjecture}
   The minimal triangulation of $\RP^4$ is unique with 16-vertices.
\end{conj}

Apart from exploring equivariant triangulations of small covers, one of the main motivations of this article is to study $\mathbb{Z}_2^4$-equivariant triangulations of $\mathbb{RP}^4$ with $16$ vertices. However, we find that no $\mathbb{Z}_2^4$-equivariant triangulation of $\mathbb{RP}^4$ exists with at most $17$ vertices (see Theorem \ref{thm7}). 

Arnoux and Marin (\cite{AM}) gave a lower bound $\frac{(n+1)(n+2)}{2}$ on the number of vertices for the triangulation of $n$-dimensional real projective space $\mathbb{RP}^n$. But this bound cannot be realized for $n \geq 4$. A series of combinatorial triangulation of $\mathbb{RP}^n$ with $2^{n+1} - 1$ vertices was constructed by K\"{u}hnel, cf. \cite{Ku}. There is a triangulation of $\mathbb{RP}^n$ with $2^n+n+1$ vertices for all $n \geq 3$, cf \cite[Theorem 3.24]{BS} and \cite[Proposition 4.3]{Kusp2023}. Recently, a new development came in  \cite{venturello2021new}, where the authors constructed a family of triangulations of the $d$-dimensional real projective space $\mathbb{RP}^d$ with at most $\Theta\bigl(\bigl(\frac{1+\sqrt{5}}{2}\bigr)^{d+1}\bigr)$ vertices for every $d \geq 1$. In \cite{Adi}, the authors constructed a triangulation of $\mathbb{RP}^n$ with at most $e^{\bigl(\frac12 + o(1)\bigr) \sqrt{n} \log n}$ vertices.
The construction of the minimal triangulations of $\mathbb{RP}^n$ for $n\geq 5$ is still unknown.

Small covers were introduced in the pioneering paper of Davis-Januszkiewicz \cite{DJ} as a generalization of real projective spaces. Briefly, an $n$-dimensional small cover is an $n$-dimensional smooth manifold with a locally standard action of $\ZZ_2^n$ where the orbit space is a simple polytope.
The results of Illman in \cite{Il} show that there is a $\ZZ_2^n$-equivariant triangulation of a small cover. An explicit proof of this is given in \cite[Lemma 3.6]{BS}. 

Inspired by the work of Davis-Januszkiewicz \cite{DJ} and Illman \cite{Il}, we ask the following.
\begin{ques}\label{qu1}
What is the  minimal $\ZZ_2^n$-equivariant triangulation of an $n$-dimensional small cover up to equivariant isomorphism?
\end{ques}
\begin{ques}\label{qu}
 Does each $\ZZ_2^n$-equivariant triangulation of an $n$-dimensional small cover $N$ with the orbit space $Q$ induce a triangulation of $Q$?
\end{ques}
In \cite{BS}, authors answered these questions when the dimension of small covers is $2$. Recently, a lower bound on the number of vertices of equivariant triangulations using equivariant covering type was given in \cite{govc2024equivariant}. In this article, we answer Question~\ref{qu} in every dimension (see Theorem~\ref{thm2}). We explicitly construct a minimal $\ZZ_2^3$-equivariant triangulation of $\RP^3$ with $11$ vertices and prove that it is unique. We also construct a $\ZZ_2^3$-equivariant triangulation of $\RP^3 \# \RP^3$ with $17$ vertices, which is the best known so far.

The article is organized as follows. In Section \ref{pritri}, we recall the definitions and some basic results on triangulations of manifolds and small covers. 

In Section \ref{sphere}, we present some results on the equivariant triangulations of $n$-spheres and characterize the equivariant triangulations of $n$-spheres with $2n+2$ vertices, proving that such a triangulation must be that of an octahedral sphere. 

In Section \ref{equitri}, we begin by proving that any $\ZZ_2^n$-equivariant triangulation of a small cover naturally induces a triangulation on its orbit space. Then, we observe some phenomena related to $\RP^n$. These results provide some  conditions for a $\ZZ_2^n$-equivariant triangulation of $\RP^n$. Following this, we construct a $\ZZ_2^3$-equivariant triangulation of $\RP^3$ with $11$ vertices from a $3$-simplex and prove that it is the unique triangulation on $11$ vertices. 

In Section \ref{sec:connectedsum}, for a finite group $G$, we present a construction method for $G$-equivariant triangulation of the connected sum of two manifolds $M$ and $N$ from their $G$-equivariant triangulations. Using this method, we provide an example of a $\ZZ_2^3$-equivariant triangulation of $\RP^3 \# \RP^3$. 

Finally, in Section \ref{sec:RP4}, we study $\ZZ_2^4$-equivariant triangulations of $\RP^4$ with at most $17$ vertices and prove that no such triangulation exists.


\section{Preliminaries}\label{pritri}
We recall some basic definitions for triangulation of manifolds following \cite{RS}.
A compact convex polyhedron which spans a subspace of dimension $n$ is called an {\em $n$-cell}. The convex hull of the standard basis of $\RR^{n+1}$ is called the {\it standard $n$-simplex}, denoted by $\Delta^n$. In this paper, by an $n$-simplex, we mean an $n$-dimensional subset of $\RR^m$ for some $m\in\mathbb{N}$, which is affinely diffeomorphic as a manifold with corners to the standard $n$-simplex.
\begin{defn}
A {\em cell complex} $X$ is a finite collection of cells in $\RR^n$ for some $n\in \mathbb{N}$ satisfying,
$(i)$ if $\sigma$ is a face of $\tau$ and $\tau \in X$, then $\sigma \in X$, $(ii)$ if $\sigma, \tau \in X$
then $\sigma \cap \tau$ is a face of both $\sigma$ and $\tau$. A cell complex $X$ is {\em simplicial} if each $\sigma \in X$ is a simplex.
\end{defn}
We denote a simplex $\sigma$ with vertices $ \{u_1, \ldots, u_k\}$ by $u_1u_2\cdots u_k$. Throughout the manuscript, we denote a standard $n$-simplex generated by the vectors $\{e_0,e_1,\ldots, e_n\}$ by $\Delta^n$, where $e_0$ is the origin and $\{e_1,\ldots,e_n\}$ is the standard basis of $\RR^n$.
The vertex set of a simplicial complex $X$ is denoted by $V(X)$. An {\em $f$-vector} of a $d$-dimensional simplicial complex $X$ is a $(d+2)$-tuple $(f_{-1}, f_0, f_1, \ldots, f_d)$ such that $f_{-1} = 1$, and for all $0 \leq i \leq d$, $f_i$ denotes the number of $i$-dimensional simplices in $X$.  For a $d$-dimensional finite simplicial complex $K$, $g_2(K)$ is defined by  $g_2:=f_1-(d+1)f_0 + \binom{d+2}{2}$. $g_2$ is one of the well known combinatorial invariant and its emergence goes back to  Walkup \cite{Walkup} where he proved that for any closed and connected triangulated $3$-manifold $K$, $g_2(K)\geq 0$. Barnette \cite{Barnette1, Barnette2, Barnette3} proved that if $K$ is the boundary complex of a simplicial $(d+1)$-polytope or, more generally, a triangulation of a connected $d$-manifold, then $g_2(K)\geq 0$. For more results on $g_2$ of a manifold and normal pseudomanifold, we refer the reader to \cite{BSR1}.

Let $X$ be a simplicial complex and $\sigma \in X$. The {\em link} of a face $\sigma$ in $X$ is defined as 
 $\{\tau \in X : \sigma \star \tau \in X ~\mbox{and}~ \sigma \cap \tau = \emptyset\}$ and is denoted by $\lk(\sigma,X)$. The {\em star} of the face $\sigma$ is defined as $\{\sigma \star \lk(\sigma,X)\}$, and is denoted by $\st(\sigma, X)$.  If $X$ is a simplicial complex, then $|X|:= \bigcup_{\sigma \in X} \sigma$ is a compact polyhedron and is referred to as the geometric carrier of $X$. For every face $\sigma$ in $X$, by $d(\sigma,X)$ (or, $d(\sigma)$ if $X$ is specified) we mean the number of vertices in $\lk (\sigma)$. For two vertices $x$ and $y$, $(x,y]$ denotes the semi-open and semi-closed edge $xy$, where $y\in (x,y]$ but $x \not \in (x,y]$. By $(x,y)$, we denote the open edge $xy$, where $x, y\not \in (x,y)$. 
 For a set $S$, we define the cardinality of $S$ as the number of elements in $S$ and is denoted by $\mathrm{Card}(S)$. The dimension of a face $\sigma\in X$ is Card($V(\sigma))-1$, and the dimension of $X$ is defined as the maximum dimension among the simplices in $X$. A triangulation  $X$ of a topological space $N$ is called minimal if $f_0(X) \leq f_0(Y)$ for all triangulations $Y$ of the topological space $N$.

\begin{defn}
{\rm If $X$, $Y$ are two simplicial complexes, then a {\em simplicial isomorphism} from $X$ to $Y$ is a bijection $\eta : X \to Y$ such that $\eta(\sigma)$ is a $k$-face of $Y$ if and only if $ \sigma$ is a $k$-face of $X$.} 
\end{defn}

We recall the definition of equivariant triangulation. This is a slight modification of the definition given in Section 2 of \cite{Il}.
\begin{defn}
{\rm Let $G$ be a finite group. A $G$-$equivariant~ triangulation$ of the $G$-space $N$ is a triangulation $X$ of $N$ such
that if $$\sigma =u_1\cdots u_k \in X, ~\mbox{then} ~ \mu(\sigma)=\mu(u_1)\cdots \mu(u_k) \in X$$ for all $ \mu\in G$, where $\mu$ is regarded as the
homeomorphism corresponding to the action of $\mu$ on $N\cong|X|$.}
\end{defn}

\begin{corollary}\label{subgroup}
If $X$ is a $G$-equivariant triangulation of an $n$-manifold $M$. Then $G$ is a subgroup of $Aut(X)$.
\end{corollary}


In the rest of this section, we recall some basics on small covers following \cite{DJ}. Let $N$ be an $n$-dimensional smooth manifold and $\rho \colon \ZZ_2^n \times \RR^n \to \RR^n$ be the standard action. An $n$-dimensional simple polytope is an $n$-dimensional convex polytope such that each of its vertex is the intersection of $n$-many codimension one faces. 
\begin{defn} (\cite{DJ})
Let $\eta \colon \ZZ_2^n \times N \to N$ be a group action. Then, $\eta$ is said to be locally standard if the following holds.
\begin{enumerate}
 \item Each $y \in N $ has a $\ZZ_2^n$-invariant open neighborhood $U_y$, that is $\eta (\ZZ_2^n \times U_y) = U_y$.
\item  There exists a diffeomorphism $\psi \colon U_y \to V$, where $V$ is a $\ZZ_2^n$-invariant (that is $\rho (\ZZ_2^n \times V) = V$) open subset of $\RR^n$.
\item There exists an isomorphism $\delta_y \colon \ZZ_2^n \to \ZZ_2^n$ such that $\psi( \eta (t, x)) = \rho(\delta_y (t), \psi(x))$ for all $(t,x) \in \ZZ_2^n \times U_y$.
\end{enumerate}
\end{defn}

\begin{defn}\cite{DJ}
A closed $n$-dimensional manifold $N$ is said to be a small cover if
there is an effective $\ZZ^n_2$-action on $N$ such that the action is a locally standard action and the orbit space of the action has the structure of a simple polytope. 
\end{defn}

\begin{example}
Let $Q_n $ be the convex-hull of the points $\{\pm e_i ~|~ i=1, \ldots, n\}$, where $e_1, \ldots, e_n$ are the standard basis of $\RR^n$. Define an equivalence relation $\sim$ on $Q_n$ by $x \sim -x$ whenever $x \in \partial Q_n$. The real projective space $\mathbb{RP}^n$ is the quotient space $Q_n/\sim $. We denote the points of $\mathbb{RP}^n$, by the equivalence classes $[x_1: \cdots : x_{n}]$ where $(x_1, \ldots, x_{n}) \in Q_n$. The standard action $\rho$ on $\RR^n$ induces an action $\rho_p \colon \ZZ_2^n \times \mathbb{RP}^n \to \mathbb{RP}^n$. Observe that this action is locally standard, and the orbit space of this action is an $n$-simplex $\Delta^n$. So, $\mathbb{RP}^n$ is a small cover. We note that 
the equivariant connected sum of two $\mathbb{RP}^n$s at fixed points is again a small cover.
\end{example}

\begin{remark}
{\rm By Proposition 1.8 of \cite{DJ}, any small cover over $\Delta^n$ is weakly equivariantly homeomorphic to $\mathbb{RP}^n$  with standard $\mathbb{Z}_2^n$-action.} 
\end{remark}

\begin{defn}\label{Charecteristic}
Let $\{\mu_1, \ldots, \mu_n\}$ be the standard generators of $\ZZ_2^n$.
Let $$X_i := \{[x_1: \cdots : x_n] \in \mathbb{RP}^n : x_i=0 ~\mbox{in}~ Q_n\}~ \mbox{for}~ i= 1, \ldots, n,$$
and $$X_0 := \{[x_1 : \cdots : x_n] \in \mathbb{RP}^n : (x_1, \ldots, x_n) \in \partial Q_n\}.$$ 
Then $X_0, X_1, \ldots, X_{n-1}$ and $ X_n$ are codimension one submanifold of $\mathbb{RP}^n$ fixed by the subgroups
$(\mu_1+ \cdots + \mu_n), (\mu_1), \ldots, (\mu_{n-1})$ and $(\mu_n)$, respectively. 

Let $\pi \colon \mathbb{RP}^n \to \Delta^n$ be the orbit map of the standard $\ZZ_2^n$-action on $\mathbb{RP}^n$. 
Let $F_i$ be the $(n-1)$-dimensional face in $\Delta^n$ not containing the vertex $e_i$ for $i=0, \ldots, n$. Then, $F_i := \pi(X_i)~\mbox{for}~ i= 0, \ldots, n$.
Thus, there is an assignment 
$$F_0 \to \mu_1+ \cdots + \mu_n:=\xi_0 ~\mbox{and} ~ F_i \to \mu_i :=\xi_i ~ \mbox{for}~ i =1, \ldots, n.$$
This assignment is called a $\ZZ_2$-{\em characteristic function} on $\Delta^n$ and $\xi_i$
is called a $\ZZ_2$-characteristic vector corresponding to the face $F_i$. 

\end{defn}

\begin{prop}\cite[Lemma 3.6]{BS}\label{equivariant}
There is a $\mathbb{Z}_2^n$-equivariant triangulation on an $n$-dimensional small cover over a simple polytope $Q$ with $2^n + f_{n-1} 2^{n-1} + \cdots + 2f_1+ f_0$ vertices, where $f_j$ is the number of $j$-dimensional faces of $Q$. 
\end{prop}

\section{Equivariant triangulation of $\Sp^n$}\label{sphere}
In this section, we present some results on the equivariant triangulations of spheres and characterize the equivariant triangulations of the $n$-sphere with $2n+2$ vertices.

Now, we discuss the natural action of $\ZZ_2^{n+1}$ on $\Sp^{n}$. We consider $$\Sp^{n} = \{(x_1, \ldots, x_{n+1}) \in \RR^{n+1} : x_1^2 + \cdots + x_{n+1}^2 =1\}$$ and
$\ZZ^{n+1}_2: =\{-1, 1\}^{n+1} \subset \RR^{n+1}$ for $n \geq 0$. The group $\ZZ^{n+1}_2$ acts
on $\RR^{n+1}$ by
\begin{equation}
 (t_1, \ldots, t_{n+1}) \times (x_1, \ldots, x_{n+1}) \to (t_1x_1, \ldots, t_{n+1}x_{n+1}).
\end{equation}
This is an effective smooth action, called standard action and denoted
by $$\rho \colon \ZZ_2^{n+1} \times \RR^{n+1} \to \RR^{n+1}.$$ The sphere $\Sp^{n}$ is invariant under this action. The orbit space $S^{n}/ \ZZ_2^{n+1}$ is $$\{(x_1, \ldots, x_{n+1}) \in \RR^{n+1} : 0 \leq x_1, \ldots,
0 \leq x_{n+1}~ \mbox{and} ~ x_1^2 + \cdots + x_{n+1}^2 =1\}$$ which is affinely diffeomorphic as manifold with corners to the $n$-simplex $\Delta^{n}$. In this section, we consider this
$\ZZ_2^{n+1}$-action on $\Sp^{n}$.

\begin{lemma}\label{lem3}
Let $X$ be a $\ZZ^{n}_2$-equivariant triangulation of $\Sp^{n-1}$ for some $n\geq 2$. Then $\mathrm{Card}(V(X))$ is even and $\mathrm{Card}( V(X)) \geq 2n$
\end{lemma}
\begin{proof}
Let $\sigma \in X$ be an $(n-1)$-simplex and  the vertex set of $\sigma$ be $\{u_1, \ldots, u_n\}$. Consider $\mu=(-1, \ldots, -1) \in \ZZ_2^n$.
Suppose that $\mu(u_i) = u_j$ for some $i, j \in \{1, \ldots, n\}$ and $i\neq j$. Let $u_iu_j$ and $\mu(u_iu_j)$ be the same edges then $u_iu_j$ must pass through origin, and thus $X$ pass through origin.
Since the $\ZZ_2^n$-equivariant triangulation of $\Sp^{n-1}$ does not pass through origin, $u_i u_j$ and $\mu(u_iu_j)$
are two distinct edges in $X$. On the other hand  $\mu^2$ is identity in $\ZZ_2^n$. This implies that $u_i= \mu(u_j)$ and
$u_iu_j \cap \mu(u_iu_j) = \{u_i, u_j\},$ which is a contradiction.
Therefore, $\mu(u_i) \notin \{u_1, \ldots, u_n\}$ for any $i=1, \ldots, n$. Since $u_i \neq u_j$ for $i \neq j$,
$\mu(u_i) \neq \mu(u_j)$ for $i \neq j$. This proves that $\mathrm{Card} (V(X)) \geq 2n$.

Further,  $X$ is an equivariant  triangulation and the cardinality of the orbit $\ZZ_2^n x$ is even for every $x\in X$, and orbits induces a partition on the $V(X)$. Therefore, $\mathrm{Card}(V(X))$ is even.  
\end{proof}
\begin{lemma}\label{missingedges}
Let $X$ be a $\ZZ_2^{n+1}$-equivariant triangulation of $\Sp^n$. If there is a vertex $v$ in $X$ with exactly $k ~(3\leq k\leq n+1)$ non-zero coordinates then there are at least $2^{k-1}(2^k-k-1)$ missing edges in $X$.
    
\end{lemma}

\begin{proof}

Without loss of generality, assume that the first $k$ coordinates of $v$ are nonzero. Let  $v = (v_1,v_2,\ldots,v_k,0,0,\ldots,0)$.  

By the group action of $\ZZ_2^{n+1}$ on the vertex $v$, we obtain $2^k$ vertices in $X$ whose first $k$ coordinates are nonzero and the remaining coordinates are zero. Let $V$ denote the set of all such vertices.  

Let $x,y \in V$ be two vertices whose coordinate signs differ at least at two positions. Suppose that $x(i)=-y(i)$ and $x(j)=-y(j)$. Let $xy$ be an edge in $X$. We consider the group element $\mu\in\ZZ_2^{n+1}$ with $\mu(i)=-1$ and $\mu(k)=1$ for $k\neq i$. By acting with $\mu$ on $xy$, we obtain an edge between two vertices $p$ and $q$, where  
$p(i)=-x(i)$, $p(k)=x(k)$ for all $k\neq i$, and $q(i)=-y(i)$, $q(k)=y(k)$ for all $k\neq i$.  

Let  
$$
xy = \{(\psi_1(t),\psi_2(t),\ldots, \psi_{i}(t),\ldots, \psi_{j}(t),\ldots, \psi_{k}(t),0,\ldots,0)\in \Sp^n : t\in[0,1]\}.
$$  
Then  
$$
pq = \{(\psi_1(t),\psi_2(t),\ldots, -\psi_{i}(t),\ldots, \psi_{j}(t),\ldots, \psi_{k}(t),0,\ldots,0)\in \Sp^n : t\in[0,1]\}.
$$  

Since $\psi_{i}:[0,1]\rightarrow \RR$ is injective and continuous, with $\psi_{i}(0)=x(i)$ and $\psi_{i}(1)=-x(i)$, there exists a unique $t'\in (0,1)$ such that $\psi_i(t')=0$. This implies that  
$$(\psi_1(t'),\psi_2(t'),\ldots, \psi_{i-1}(t'),0,\psi_{i+1}(t'),\ldots, \psi_{j}(t'),\ldots, \psi_{k}(t'),0,\ldots,0)\in xy\cap pq.$$  
This is a contradiction because $\{x,y\}\cap\{p,q\}=\emptyset$, and therefore $xy\cap pq$ must be empty in $X$.  

Similarly, if $pq$ is an edge, then using the group element $\epsilon\in \ZZ_2^{n+1}$ with $\epsilon(j)=-1$ and $\epsilon(k)=1$ for $k\neq j$, we find that $\epsilon(pq)=xy$. Hence, the edges $xy$ and $pq$ must be missing edges in $X$.  

This implies that each vertex in $V$ can be joined only with those vertices in $V$ whose coordinate signs differ exactly at one position. Thus, each vertex can be joined with at most $k$ vertices in $V$. Therefore, there can be at most $k2^{k-1}$ edges among the vertices of $V$. Hence, there are at least  
$\binom{2^k}{2}-k2^{k-1} \;=\; 2^{k-1}(2^k-k-1)$  missing edges in $X$.
\end{proof}

\begin{lemma}\label{lem6}
 Fix $n\geq 2$. There is no $\ZZ_2^{n+1}$-equivariant triangulation of $\Sp^n$ with $2^{n+1}$ vertices with a
vertex $(x_1,x_2,\dots,x_{n+1})$ such that $x_1,x_2,\dots, x_{n+1}$ are non zero.
\end{lemma}
\begin{proof}
Let $X$ be a $\ZZ_2^{n+1}$-equivariant triangulation of $\Sp^n$ with $2^{n+1}$ vertices such that
$(x_1,x_2,\dots,x_{n+1})\in V(X)$ and $x_1, x_2, \dots,  x_{n+1}$ are non zero.  Then from Lemma \ref{missingedges}, $X$ has at least $2^{n}(2^{n+1}-n-2)$ missing edges. Thus, at most $(n+1)2^n$ edges. Since $g_2(X)\geq 0$, $f_1(X)\geq (n+1)f_0(X)-\binom{n+2}{2}$.
Now, we have 
\begin{align*}
  (n+1)f_0(X)-\binom{n+2}{2} \leq  f_1(X)\leq & (n+1)2^n\\
  (n+1)2^{n+1}-\binom{n+2}{2}\leq & (n+1)2^n\\
  2^n((n+1)2-1)\leq &\binom{n+2}{2}\\
  2^n(2n+1)\leq &\binom{n+2}{2}. \\
    \end{align*}
   The last inequality  does not hold for $n\geq 2$. Therefore,  $X$ does not exist.
\end{proof}

\begin{prop}\label{suspension}
Let $X$ be a $\ZZ_2^{n+1}$-equivariant triangulation of an $n$-sphere such that $\lk(u, K) = \lk(v, K)$ and $\mu(\st(u, K)) \cong \st(v, K)$ for some $u, v \in V(X)$ and $\mu\in \ZZ_2^{n+1}$. Then we have $X = \lk(u,X) \star \{u, v\}$.
\end{prop}

\begin{lemma}\label{twononzero}
    Let $X$ be a $\ZZ^{n}_2$-equivariant triangulation of $\Sp^{n-1}$ for some $n\geq 2$. If $\mathrm{Card}(V(X))=2n$ then for any vertex $x\in V(X)$, at most two coordinates of $x$ are non-zero. Moreover, the number of choices for the $V(X)$ is $\lfloor\frac{n}{2}\rfloor +1$.
\end{lemma}

\begin{proof}
Fix $i \in \{1,2,\dots,n\}$  such that $x_i = 0$ for every vertex $x = (x_1, \ldots, x_n) \in V(X)$.  
We claim that for every $y \in |X|$, $y_i = 0$. On contrary suppose that there exists $z \in |X|$ such that $z_i \neq 0$.  
Then there must be a simplex $\sigma \in X$ such that $z \in |\sigma|$.  But all the vertices of $\sigma$ have $i$-th coordinate zero.  
Since the vertices of $\sigma$ are affinely independent, they lie in the hyperplane $x_i = 0$.  
Hence, any point in $|\sigma|$ must also have $i$-th coordinate zero, contradicting $z_i \neq 0$.  
Thus, $|X| \subseteq \mathbb{R}^{n-1}$, which is impossible.  
Therefore, for each $i \in \{1,2,\dots,n\}$, there exists a vertex $v \in X$ such that $v_i \neq 0$.

Now, consider a vertex $u \in V(X)$ with $k \ge 3$ non-zero coordinates and $n-k$ coordinates equal to zero.  
Without loss of generality, assume that $u_i \neq 0$ for $1 \le i \le k$ and $u_j = 0$ for $k+1 \le j \le n$.  
By the action of $\mathbb{Z}_2^n$ on $u$, we obtain $2^k$ vertices in $X$.  

On the other hand, for each $k+1 \le j \le n$, there exists a vertex $v$ such that $v_j \neq 0$.  
Again, by the action of $\mathbb{Z}_2^n$ on such vertices, we obtain at least $2(n-k)$ vertices in $X$.  Hence, $X$ has at least $2^k + 2(n-k) > 2n$ vertices if $k \ge 3$, which is a contradiction. Therefore, for every vertex $x \in X$, at most two coordinates of $x$ are non-zero.

One can observe from above that the number of choices for $V(X)$ is same as the number of partitions of $2n$ using only $2$'s and $4$'s.
Let $k$ be the number of $4$'s used in the partition. Each $4$ contributes $4$ to the total, so the remaining part $2n - 4k$ must be covered by $2$'s.  
The number of $2$'s used is then 
$m = \frac{2n - 4k}{2} = n - 2k.$
Since $m \ge 0$, we have $n - 2k \ge 0$, which implies $k \le \lfloor n/2 \rfloor$.  
Also, $k \ge 0$ by definition. Thus, $k$ can take all integer values from $0$ to $\lfloor n/2 \rfloor$. For each choice of $k$, there is exactly one corresponding partition consisting of $k$ $4$'s and $n-2k$ $2$'s. Therefore, the total number of partitions of $2n$ using only $2$'s and $4$'s is $\lfloor n/2 \rfloor + 1.$  
\end{proof}

\begin{theorem}\label{2nvertexsphere}
Let $X$ be a $\ZZ^{n}_2$-equivariant triangulation of $\Sp^{n-1}$ for some $n\geq 2$. If $\mathrm{Card}(V(X))=2n$ then $X$ is an octahedral sphere.
\end{theorem}

\begin{proof}
We prove the result by applying induction on $n$.  First we assume that $n=2$. Then $X$ has four vertices with two pairs of antipodal vertices. Since there cannot be an edge between antipodal vertices and it must be a cycle, the result follows. 

Let the result be true for $n=1,2,\dots,m-1$. Let $X$ be a $\ZZ^{m}_2$-equivariant triangulation of $\Sp^{m-1}$ with $2m$ vertices and $m\geq 3$. From Lemma~\ref{twononzero}, it follows that for any $x \in V(X)$, the vector $x$ has at most two non-zero coordinates.

Let $\{\{u^1, v^1\}, \ldots, \{u^k, v^k\}\}$ be the collection of antipodal pairs of vertices in $V(X)$ each having exactly one non-zero coordinate. Then there must be $2m - 2k$ vertices in $V(X)$ having exactly two non-zero coordinates. Let there exist a vertex $u^1 \in X$ such that $u_i^1 \neq 0$ for some $i \in \{1,2,\ldots,n\}$ and $u_j^1 = 0$ for all $j \neq i$. Define $v^1 = -u^1$. Then $v^1 \in V(X)$ with $v_i^1 = -u_i^1 \neq 0$ and $v_j^1 = 0$ for all $j \neq i$.

We claim that $\lk(u^1, X) = \lk(v^1, X)$.

Note that $u^1 v^1 \notin X$. Let $\sigma$ be a $k$-simplex in $\lk(u^1, X)$ for some $k \geq 1$. Choose a group element $\mu \in \ZZ_2^n$ such that $\mu_i = -1$ and $\mu_j = 1$ for all $j \neq i$. Then $\mu(\sigma \star u^1) = \sigma \star v^1$. Hence $\sigma$ is a $k$-simplex in $\lk(v^1, X)$, and thus $\lk(u^1, X) \subseteq \lk(v^1, X)$. Similarly, $\lk(v^1, X) \subseteq \lk(u^1, X)$. Therefore, $\lk(u^1, X) = \lk(v^1, X)$, proving our claim. From Proposition \ref{suspension}, $X=\lk(u^1)\star\{u^1,v^1\}$ where $\lk(u^1,X)$ is a $\ZZ^{m-1}_2$-equivariant triangulation of $\Sp^{m-2}$. Thus, from the induction hypothesis, $\lk(u^1,X)$ is an octahedral $(m-2)$-sphere and therefore, $X$ is an octahedral $(m-1)$-sphere.

Now, we assume that all the vertices of $X$
has exactly two non-zero coordinates.
Now, let $x^1 \in V(X)$ be a vertex having exactly two non-zero coordinates. Suppose that $x_i^1, x_j^1 \neq 0$ for some $i, j \in \{1, 2, \ldots, m\}$ and $x_l^1 = 0$ for all $l \neq i, j$. Let $x^2 = -x^1$ is the antipodal vertex in $X$. Define $x^3, x^4$ such that $x_i^3 = x_i^1$, $x_j^3 = -x_j^1$, and $x_l^3 = 0$, and $x_i^4 = -x_i^1$, $x_j^4 = x_j^1$, and $x_l^4 = 0$ for all $l \neq i,j$. Then $x^3, x^4 \in V(X)$ and they form an antipodal pair. The vertex set $V(X) \setminus \{x^1, x^2, x^3, x^4\}$ contains exactly $m - 2$ pairs of antipodal vertices, and a maximal simplex in $\lk(x^1, X)$ contains $m  - 1$ vertices. Therefore, $x^1$ must be joined with either $x^3$ or $x^4$ in $X$. Suppose that $x^1 x^3 \in X$.

We claim that $x^1 x^4 \in X$. Let $\gamma$ be a maximal simplex in $\lk(x^1, X)$ containing $x^3$. Then the $(m - 3)$-simplex $\gamma \setminus \{x^3\}$ consists of exactly $m - 2$ vertices, each corresponding to one of the remaining antipodal pairs in $V(X) \setminus \{x^1, x^2, x^3, x^4\}$. Since $\gamma \setminus \{x^3\}$ must be joined with two vertices in $\lk(x^1, X)$, the only possibilities are $x^3$ and $x^4$. Hence, $x^1 x^4 \in X$. This proves the claim. Similarly, $x^2 x^3, x^2 x^4 \in X$.

We consider the group element $\mu' \in \ZZ_2^n$ such that $\mu'_i = \mu'_j = -1$ and $\mu'_l = 1$ for all $l \neq i, j$. 

Let $\beta$ be a maximal simplex in $\lk(x^1, X)$. Then $\beta$ must contain either $x^3$ or $x^4$. Suppose that $\beta$ contains $x^3$, so $\beta = \beta' \star x^3$ for some simplex $\beta'$ in $\lk(x^1, X)$. Since $\beta'$ must also be joined with $x^4$ in $\lk(x^1, X)$, we have $\mu'(\beta' \star x^4 \star x^1) = \beta' \star x^3 \star x^2$. Thus, $\beta \in \lk(x^2, X)$. Similarly, every maximal simplex $\alpha$ containing $x^4$ in $\lk(x^1, X)$ also lies in $\lk(x^2, X)$. Hence, $\lk(x^1, X) \subseteq \lk(x^2, X)$. By symmetry, $\lk(x^2, X) \subseteq \lk(x^1, X)$, so $\lk(x^1, X) = \lk(x^2, X)$.
Therefore, from Proposition \ref{suspension}, $X=\lk(x^1,X)\star\{x^1,x^2\}$ where $\lk(x^1,X)$ is a $\ZZ^{m-1}_2$-equivariant triangulation of $\Sp^{m-2}$ with exactly $2m-2$ vertices. Thus, from the induction hypothesis, $\lk(x^1,X)$ is an octahedral $(m-2)$-sphere. Therefore, $X$ is an octahedral $(m-1)$-sphere. 
\end{proof}

\begin{corollary}\label{8vertex}
Let $X$ be a $\mathbb{Z}_2^4$-equivariant triangulation of $\mathbb{S}^3$ with $8$ vertices. 
Then the only possibility for the vertex set $V(X)$ is (up to permutation of coordinates) one of the following for some fixed $a,b,c,d \in\mathbb{R}\setminus\{0\}$.
$$
\begin{aligned}
&(I)\quad  V(X) = \{(\pm a,\pm b,0,0),\ (0,0,\pm c,\pm d)\} \subset \mathbb{S}^3.\\
&(II)\quad V(X) = \{(\pm a,\pm b,0,0),\ (0,0,\pm 1,0),\ (0,0,0,\pm 1)\} \subset \mathbb{S}^3.\\
&(III)\quad V(X) = \{(\pm 1,0,0,0),\ (0,\pm 1,0,0),\ (0,0,\pm 1,0),\ (0,0,0,\pm 1)\}.
\end{aligned}
$$
\end{corollary}

\begin{proof}
The proof follows from Lemma \ref{twononzero} and Theorem \ref{2nvertexsphere}.
\end{proof}

\begin{remark}
    Let $X$ be a $\mathbb{Z}_2^n$–equivariant triangulation of $\mathbb{S}^{n-1}$. Then the suspension of $X$ is a $\mathbb{Z}_2^{n+1}$–equivariant triangulation of $\mathbb{S}^{n}$. This  establishes a method to construct a $\mathbb{Z}_2^{n+1}$–equivariant triangulation of $\mathbb{S}^n$ from a given $\mathbb{Z}_2^{\,n}$–equivariant triangulation of $\mathbb{S}^{n-1}$.
\end{remark}

\begin{example}\label{trid3}
\begin{figure}[ht]
\centerline{
\scalebox{.80}{
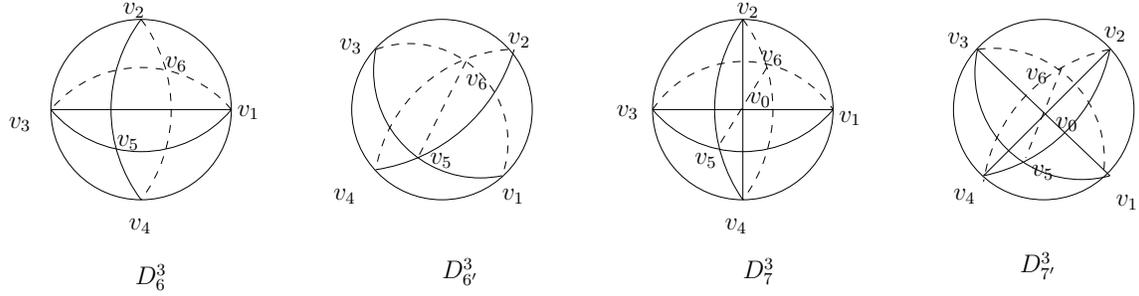
 }
 }
\caption {$\ZZ_2^3$-equivariant triangulations of $D^3$ with $6$ and $7$ vertices.}
\label{egch201a}
\end{figure}
We consider the $3$-dimensional disc $$\mathbb{D}^3=\{(x_1, x_2, x_3) \in \RR^3 : x_1^2 + x_2^2+x_3^2 \leq 1\}.$$ 
We discuss all the $\mathbb{Z}_2^3$-equivariant triangulations of $\mathbb{D}^3$ with vertices $\leq 7$. Let $\mu_1, \mu_2, \mu_3$ be the standard generators of $\mathbb{Z}_2^3$. That is, $\mu_i$ acts on $\mathbb{D}^3$ by reflection along the plane $\{x_i = 0\}$ for $i = 1, 2, 3$. Let $Y$ be a $\mathbb{Z}_2^3$-equivariant triangulation of $\mathbb{D}^3$.  
Then $\partial Y$ is a $\mathbb{Z}_2^3$-equivariant triangulation of $\Sp^2$.  
So $Y$ contains at least $6$ vertices by Lemma \ref{lem3}. 

Let $Y$ contain $6$ vertices.  
Then $(0,0,0)$ cannot be a vertex of $Y$.  
Let $u_1u_2u_3u_4 \subset D^3$ be a $3$-simplex in $Y$ such that $(0,0,0) \in |u_1u_2u_3u_4|$.  
Since $Y$ is equivariant, $(0,0,0)$ cannot belong to the interior of a face $F$ of $u_1u_2u_3u_4$ with $\dim(F) \geq 2$.  Thus $(0,0,0)$ belongs to the interior of some edge, say $u_1u_2$. Then $\mu_i(u_1u_2)$ is an edge in $Y$ for $i = 1, 2, 3$. If $u_1,u_2$ do not lie on axes, then $\mu_i(u_1u_2)$ will be different edge for $1\leq i\leq 3$ and  $\mu_i(u_1u_2)\cap u_1u_2=(0,0,0)$. This implies that $(0,0,0)$ is a vertex in $Y$, which is a contradiction.
Therefore $u_1, u_2$ must belong to the axis $\{x_i\}$ for some $i \in \{1,2,3\}$.  Since $Y$ is $\mathbb{Z}_2^3$-equivariant, if $u_1, u_2$ belong to the axis $x_i$, then $u_3, u_4 \in Y \cap \{x_i = 0\}$.  

Let $Y$ contain $7$ vertices.  
Then $(0,0,0) \in V(Y)$ and all other vertices belong to $\partial Y$.  
Let $u = (a,b,c) \neq (0,0,0)$ be a vertex of $Y$.  It is clear to observe that at least one $a,b$, and $c$ must be zero.

Observing these facts, we present all possible $\mathbb{Z}_2^3$-equivariant triangulations of $\mathbb{D}^3$ with $6$ and $7$ vertices in Figure $3$.  
We denote these triangulations by $D^3_6$, $D^{3}_{6^{\prime}}$, $D^3_7$ and $D^{3}_{7^{\prime}}$, respectively.  

\end{example}

\section{Equivariant triangulation of a class of small covers}\label{equitri}
In this section, we first prove that any equivariant triangulation of a small cover induces a triangulation on the corresponding orbit space. Then, we establish some useful lemmas for the study of equivariant triangulations of $\RP^n$. Finally, we present an example constructing an equivariant triangulation of $\RP^3$ with $11$ vertices and prove that it is the unique minimal equivariant triangulation on $11$ vertices.

\begin{lemma}\label{lem2}
Let $\sigma$ be an $n$-simplex in $\RR^{n+1}$ and $\mu$ be a reflection on the hyperplane $H$ in $\RR^{n+1}$
such that $\mu(\sigma)=\sigma$. Then $H$ contains exactly $n-1$ vertices of $\sigma$. Moreover, there
exists two $n$-simplices $\sigma_1$ and $\sigma_1^{\prime}$ such that $\sigma = \sigma_1 \cup \sigma^{\prime}_1$ and $\mu(\sigma_1)=\sigma_1^{\prime}$.
\end{lemma}
\begin{proof}
Without loss of generality, we assume that $\sigma$ is the standard $n$-simplex.
Since $\sigma$ is an $n$-simplex, $\mu(\sigma)=\sigma$ and $H$ is a hyperplane in $\RR^{n+1}$, $H$ cannot contain more than $n-1$ vertices. Thus, the result follows when $n=1,2$ (see Figure \ref{rpn206}). 
\begin{figure}[ht]
\centerline{
\scalebox{0.88}{
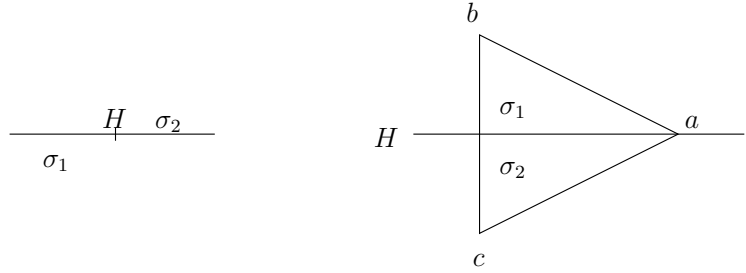
 }  }
\caption {Reflections along a hyperplane.}
\label{rpn206}
\end{figure}

Let $n > 2$ and $\sigma=u_1\cdots u_{n+1}$. Suppose that $H$ contains less than $n-1$ vertices of $\sigma$. Without loss of generality, we assume that $H$ does not contain the vertices $\{u_1, \ldots, u_k\}$ of $\sigma$. Then $k \geq 3$.
 Since $\mu(u_i) \neq u_i$ for all $i = 1, \ldots, k$, $k$ is even and it follows that $k \geq 4$. Let $H_{>}$
and $H_{<}$ be the open half spaces of $H$ in $\RR^{n+1}$. Then half of the vertices of $\{u_1, \ldots,
u_k\}$ belong to $H_{>}$. Let the vertices $u_{i_1}, \ldots, u_{i_l}$ of $\sigma$ belong to $H_{>}$ for some $l\geq 2$. Since $u_{i_1}u_{i_l}$ is an edge in $\sigma$, and $\mu$ is a reflection, $\mu(u_{i_1}u_{i_l})$ is an edge of $\sigma$ in $H_{<}$.
Thus, $u_{i_1}u_{i_l}$ and $\mu(u_{i_1}u_{i_l})$ form a $2$-dimensional affine subspace. This contradicts the fact that $\sigma$ is an $n$-simplex. This proves that $H$ contains exactly $n-1$ vertices of $\sigma$.

Let $\tau =\sigma \cap H_{>}, ~ \tau^{\prime} = \sigma \cap H_{<}, ~ \sigma_1= \bar{\tau} ~\mbox{and} ~ \sigma_1^{\prime}= \bar{\tau^{\prime}}$, where $\bar{\tau}$ and $\bar{\tau'}$ represent the closure of $\tau$ and $\tau'$, respectively. Then $\sigma_1$ and $\sigma_1^{\prime}$ are
$n$-dimensional simplices in $\RR^{n+1}$ with $\sigma= \sigma_1 \cup \sigma_1^{\prime}$ and $\mu(\sigma_1)
=\sigma_1^{\prime}$.
\end{proof}

\begin{lemma}\label{lem_simp_ref}
Let $\Delta^k \subset \mathbb{R}^{m}$ be a $k$-simplex such that the origin belongs to the interior of $\Delta^k$ and $\mu_i(\Delta^k)=\Delta^k$ for each $i$-th reflection $
\mu_i$ on $x_i=0$. Then $k=1$, and $|\Delta^k|$ is a subset of an axis. 
\end{lemma}
\begin{proof}
If $v \in V(\Delta^k)$, then $-v \in V(\Delta^k)$, and $v(-v)$ is an edge of $\Delta^k$. Suppose that $k>1$. Then, there are two edges of $\Delta^k$ of the form $v(-v)$ intersecting at the origin. So, the origin is a vertex of $\Delta^k$, a contradiction. Hence, $\Delta^k=v(-v)$ for some $v\in \RR^m$. 

Suppose that  $v$ has $\ell >1$ non-zero coordinates. Then, the convex hull of $\{\mu_i(v): i=1, \ldots, \ell\}$ is an $\ell$-cube. So, it cannot be a face of $\Delta$. Thus, each vertex of $\Delta$ can have only one non-zero coordinate, as the origin is not a vertex of $\Delta^k$. Therefore,  $\Delta^k=v(-v)$ is a subset of an eaxis.
\end{proof}

\begin{theorem}\label{thm2}
Let $M$ be an $n$-dimensional small cover.
Let $X$ be a $\ZZ_2^n$-equivariant triangulation of $M$ and $\pi \colon |X|\rightarrow P$ is the orbit map where $P$ is a simple polytope. Then the collection
$\{\pi(\sigma) \colon \sigma \in X ~ \mbox{and}~ \dim{\sigma} = n\}$ gives a triangulation of the polytope $P$.
\end{theorem}

\begin{proof}
Let $M$ be the $n$-dimensional small cover and $X$ be a $\ZZ_2^n$-equivariant triangulation of $M$. Then the orbit space of $|X|$ is a polytope $P$. Thus, we have the following diagram.
\begin{center}
  
$
\begin{tikzcd}[row sep=large, column sep=large]
  & {|X|}\arrow[dr, "\pi"] \arrow[dl, "\cong"'] & \\
  M=(\mathbb{Z}_2^n \times P)/{\sim} \arrow[rr, "\pi"'] & & P
\end{tikzcd}
$
\end{center}
Let $\mu \in \ZZ_2^n$. Define
$P_{\mu} := \{[\mu, p]_{\sim} \in |X| : p \in P\}$.
Then $P_{\mu} \cong P$ as manifold with corners. Let $\sigma$ be a simplex in $X$. Since $X$ is a $\ZZ_2^n$–equivariant triangulation of $M$, there exists $\mu \in \ZZ_2^n$ such that $|\sigma| \cap P_{\mu} \neq \emptyset$. This implies that
$|\mu^{-1}(|\sigma|)| \cap P_{\bf 0} \neq \emptyset$ where ${\bf 0}\in \ZZ_2^n$ is the identity element. For each $p \in P_{\bf 0}$, there exists a unique simplex $\tau_p \in X$ such that the relative interior of $|\tau_p|$ contains $p$. Define 
$X_{\bf 0} = \{\tau_p : p \in P_{\bf 0} \neq \emptyset\}$.
Observe that

$\bigcup_{\mu \in \ZZ_2^n} \mu(X_{\bf 0}) = \{\mu(\tau) : \tau \in X_{\bf 0}\} = X$, and $P_{\bf 0} \subseteq |X_{\bf 0}|$. 

\noindent\textbf{Claim.} If $\tau \in X_{\bf 0}$ is a simplex, then $|\tau| \cap P_{\bf 0}$ is also a simplex.

If $|\tau| \subset P_{\bf 0}$, then done. Otherwise, there are some points in the relative interior of $|\tau|$ which do not belong to $P_{\bf 0}$. Let $F_{i_1}$ be a facet of $P_{\bf 0}$ such that $F_{i_1}$ contains some interior point of $|\tau|$. Then, $\xi_{i_1}(\tau)=\tau$. 
Then, by Lemma~\ref{lem2}, there exist simplices $\tau_1$ and $\tau_1'$ such that $\xi_{i_1}(\tau_1) = \tau_1'$, $|\tau| = |\tau_1| \cup |\tau_1'|$, and $\pi(|\tau|) = \pi(|\tau_1|)$.

Now, if  $|\tau_1| \subset P_{\bf 0}$, then done. Otherwise, there are some points in the relative interior of $|\tau_1|$ which do not belong to $P_{\bf 0}$. Let $F_{i_2}$ be a facet of $P_{\bf 0}$ such that $F_{i_2}$ contains some interior point of $|\tau_1|$.
Then, $\xi_{i_2}(\tau_1)=\tau_1$. Then, by Lemma~\ref{lem2} there exist simplices $\tau_2$ and $\tau_2'$ such that $\xi_{i_2}(\tau_2) = \tau_2'$, $|\tau_1| = |\tau_2| \cup |\tau_2'|$, and $\pi(|\tau|) = \pi(|\tau_1|)=\pi(|\tau_2|)$.
Proceeding in this way, after a finitely many steps, we obtain simplices $\tau, \tau_1, \ldots, \tau_k$ such that
$\pi(|\tau|) = \pi(|\tau_1|) = \cdots = \pi(|\tau_k|) = |\tau_k|$.

Next, consider the collection
$\{\tau_k \text{ and all its faces} : |\tau_k| = \pi(|\tau|) ~\mbox{ for some } \tau \in X_{\bf 0}\}$. Since $\pi$ is the orbit map, 
 $\pi(\sigma) \cap \pi(\sigma') = \pi(\sigma \cap \sigma')$
for any $\sigma,\sigma' \in X$.   Thus, the collection
$\{\tau_k \text{ and all its faces} : |\tau_k| = \pi(|\tau|) ~\mbox{ for some } \tau \in X_{\bf 0}\}$  is a simplicial complex whose geometric carrier is homeomorphic to $P$.
\end{proof}

\begin{lemma}
Let $M$ be an $n$-dimensional small cover, and let $X$ be a $\mathbb{Z}_2^n$-equivariant triangulation of $M$. Let $\pi \colon |X| \rightarrow P$ be the orbit map, where $P$ is a simple polytope. Then $f_0(X) \geq 2n$.
\end{lemma}

\begin{proof}
Let $e \in P$ be a vertex. Then $\pi^{-1}(e) \subseteq |X|$ is a set of fixed vertices. Define 
$D_e = \{\sigma \in X : e \in |\sigma|\}.$
Then $D_e$ is a triangulated $(n-1)$-ball, and its boundary $\partial D_e$ is a $\mathbb{Z}_2^n$-equivariant triangulation of $\mathbb{S}^{n-1}$. By Lemma~\ref{lem3}, $\partial D_e$ contains at least $2n$ vertices, and therefore $X$ has at least $2n$ vertices.
\end{proof}

\begin{lemma}\label{lem1}
Let $X$ be a $\ZZ_2^n$-equivariant triangulation of $\mathbb{RP}^{n}$. If $e_0, \ldots, e_{n}$ are fixed points of $|X|$, then $\mathrm{Card}(V(X)\setminus \{e_i : i=0, \ldots, n\})$ is even. 
\end{lemma}
\begin{proof}
Let $A = V(X)\setminus \{e_i : i= 0, \ldots, n\}$ and $u \in A$. Since $\ZZ_2^n$ action on $\mathbb{RP}^n$
is standard, the isotropy group of $u$ is isomorphic to $\ZZ_2^k$ for some $0 < k \leq n$. Thus, 
the cardinality of the set $\{\mu u : \mu\in \ZZ_2^n\}$ is  even. Since $X$ is a $\ZZ_2^n$-equivariant
triangulation, $\{\mu u : \mu\in \ZZ_2^n ~~\mbox{and}~~ u \in A\} = A$ and the sets $\{\mu u : \mu\in \ZZ_2^n\}$ make a partition of $A$. Hence $A$ contains even number of vertices.
\end{proof}

\begin{lemma}\label{lem4}
Let the notations be same as in Definition 
\ref{Charecteristic}. Let $X$ be a $\ZZ_2^n$-equivariant triangulation of $\mathbb{RP}^n$ and $\sigma$ be an $n$-simplex
in $X$ such that $e_i \in |\sigma|$ for some $i \in \{0, \ldots, n\}$. Then $V(\sigma) \cap
X_i$ is empty.
\end{lemma}
\begin{proof}
First suppose that $e_i$ is a vertex of $\sigma$ and $V(\sigma) \cap X_i \neq \empty$ with $u \in V(\sigma)
\cap X_i$. Now, we claim that  $e_iu$ and $\xi_i(e_iu)$ are two distinct edges in $X$, where $\xi_i$ is the $\ZZ_2$ characteristic vector (see Definition \ref{Charecteristic}). Suppose $\xi_i(ue_i)=ue_i$. Since $\xi_i$ only fixes the set $X_i\cup\{e_i\}$, $\xi_i(e_i)=e_i$, $\xi_i(u)=u$. This implies that $\xi_i(x)=x$ for all $x\in (u,e_i)$. Thus, $[u,e_i)\subseteq X_i$. Therefore, $\pi(ue_i)$ is a disconnected subset of $\Delta^n$ whereas $ue_i$ is connected in $X$. This is a contradiction. Therefore, $e_iu$ and $\xi_i(e_iu)$ are two distinct edges in $X$ with $e_iu \cap \xi_i(e_iu)
=\{e_i, u\}$ in $X$. This is a contradiction. 

Suppose that $e_i$ is an interior point of a $k$-dimensional face $\tau$ of $\sigma$ with $k >1$. Since $e_i$ is a fixed point, either $\mu(\tau) =\tau$ or $\mu(\tau) \cap \tau$ is a proper face of $\tau$ for any $\mu\in \ZZ_2^n$. Since $e_i$ is an interior point of $\tau$, $\mu(\tau) \cap \tau$ is not a proper face of $\tau$ for any $\mu\in \ZZ_2^n$. So $\mu(\tau) = \tau$ for all $\mu\in \ZZ_2^n$. Then $\partial \tau$ is a $\ZZ_2^n$-equivariant $(k-1)$-dimensional simplicial sphere. Using Lemma \ref{lem3}, we get that $\mathrm{Card}(V(\partial \tau)) \geq 2k > k+1, ~~\mbox{if} ~~ k >1.$  
Hence $e_i$ can not belong to the interior of a face $\tau$ of $\sigma$ with $\mathrm{dim}(\tau) \geq 2$.
Thus, $e_i$ belongs to the interior of an edge, namely $u_ku_l$, of $\sigma$.

If $u_k=u$, then since $e_i$ is an interior point of $uu_l$, $\xi_i(u_l)=u_l$. In this case, $uu_l\setminus e_i \subseteq X_i$. Then $\pi(uu_l)$  is a disconnected subset of $\Delta^n$, which is a contradiction. Thus $u_k\neq u$ and similarly, $u_l\neq u$.  Since $X$ is $\ZZ_2^n$-equivariant and $e_i$ is an interior point of the edge $u_ku_l$, we have two choices: $\xi_i(u_k)
= u_l$ or $\xi_i(u_k)=u_k$. From the arguments given above, $\xi_i(u_k)=u_k$ is not possible. Thus, $\xi_i(u_k)=u_l$.
Also, we have $u_ku_l \subset X_{i_1} \cap \ldots \cap X_{i_{n-1}}$ for some $\{i_1, \ldots, i_{n-1}\} \subset \{0, \ldots, \hat{i}, \ldots, n\}$. This implies that $\xi_i(uu_k)=uu_l$. We have $\xi_i(u_ke_i)=e_iu_l$. Thus, there exists a line segment, say, $L$ between $u$ and $e_i$ such that $\xi_i$ is a reflection along $L$ and $\xi_i(L)=L$. It follows that $\xi_i(ue_i)=ue_i$ with $\xi_i(u)=u$ and $\xi_i(e_i)=e_i$. This implies that $[u,e_i)\subseteq X_i$. Consequently, $\pi(ue_i)=\pi(L)$ is disconnected. This is a contradiction. Hence, $V(\sigma)\cap X_i$ is empty.
\end{proof}

\begin{lemma}\label{lem5}
 Let the notations be same as in Definition \ref{Charecteristic}. Let $X$ be a $\ZZ_2^n$-equivariant triangulation of $\mathbb{RP}^n$ such that the intersection $V(X) \cap \pi^{-1}((e_i,e_j))$ is empty for some $i, j \in \{0, \ldots, n\}$ with $i < j$. Then there is a 2-dimensional face $F \subset \Delta^n$ containing $e_ie_j$ with $V(X) \cap \pi^{-1}(F\setminus V(F))$ is nonempty, when $n>2$.
\end{lemma}
\begin{proof}
We assume that the $\ZZ_2^n$-action  is standard on $\mathbb{RP}^n$ and $(i, j)= (0, 1)$. Let $$D_0 = \{\sigma \in X : e_0 \in |\sigma|\}.$$
By Lemma \ref{lem4}, $V(D_0) \cap X_0$ is empty. Let $n > 2$. 
We observe that $|D_0| \cap \pi^{-1}([e_0, e_1))$
is a curve segment, say $u_0u_1$, with points $u_0$ and $u_1$ and $u_0, u_1 \in |\partial{D_0}|$. Note that $u_0$ and $u_1$ cannot be in $V(X)$ by the assumption. So, there exists a $k$-simplex $\sigma \in \partial{D_0}$ of smallest dimension
such that $u_0 $ belongs to the interior of $ {|\sigma|}$ and $k > 0$.

The segment $u_0u_1$ is fixed by $\mu_2, \ldots, \mu_n$, where
$\{\mu_1, \ldots, \mu_n\}$ are the standard generators of $\ZZ_2^n$. Since $\partial{D_0}$ is $\ZZ_2^n$-equivariant, $\mu_i(\sigma)$ is also a $k$-simplex in $\partial{D_0}$ and $\mu_i(\sigma)=\sigma$, for $i=2, \ldots, n$. We can use Lemma \ref{lem_simp_ref} (up to an equivariant homeomorphism), to get that $\sigma$ is an edge in $X$, and $V(\sigma) \cap \pi^{-1}(F\setminus V(F))$ is nonempty for a 2-dimensional face $F$ containing $e_0e_1$. 

\end{proof}

\begin{remark}\label{remark}
 Let $X$ be a $\ZZ_2^n$-equivariant triangulation of $\mathbb{RP}^n$ and $x \in V(X)$. If $\pi(x)=y$ belongs to the relative interior of a $k$-dimensional face of $\Delta^n$,
then $\pi^{-1}(y) = \{\mu(x): \mu\in \ZZ_2^n\} \subset V(X)$ and $\mathrm{Card}(\{\mu(x): \mu\in \ZZ_2^n\}) = 2^k$. 
\end{remark}

Let $C(x_1,\dots,x_m)$ denote a cycle of length $m$ with vertex set $\{x_1,\dots,x_m\}$. 
Let $B_{x_1,\dots,x_m}(p;q)$ denote the bi-pyramid with base vertices $x_1,\dots,x_m$ and apexes $p$ and $q$.

\begin{defn}\label{defn}
Let $K$ be a triangulated $3$-manifold and let $w$ be a vertex of $K$. 
Assume that the link of $w$ in $K$ is given by
$\lk(w,K)=B_{x_1,\dots,x_m}(p;q),$
for some vertices $x_1,\dots,x_m,p,q\in K$ with $pq\notin K$. Define $K' = \bigl(K \setminus \{\alpha \in K : w \in \alpha\}\bigr) \cup \{\, pq \star C(x_1,\dots,x_m) \,\}.$
Then $|K'|\cong |K|$, and we say that $K'$ is obtained from $K$ by a \emph{retriangulation} at $\operatorname{st}(w,K)$ with apexes $p$ and $q$.
\end{defn}

\begin{example}\label{rp3e}
In this example, we construct a $\ZZ_2^3$-equivariant triangulation of $\mathbb{RP}^3$  on $11$ vertices. 
Consider the cross-polytope $P^3$ in $\RR^3$. Then $\mathbb{RP}^3$
can be obtained by identifying the antipodal points of the boundary of $P^3$. Let  $B_1 = \{(x_1, x_2, x_3) \in P^3 : x_3 \geq 0\} \quad
 \mbox{and} \quad B_2 = \{(x_1, x_2, x_3) \in P^3 : x_3 \leq 0\}$. We can see that $P^3=B_1\cup B_2$.  Let $V(B_1)=\{1, 2, -1, -2, 3\}$ and $V(B_2)=\{1, 2, -1, 2, -3\}$. We construct a triangulation of $B_1$ and $B_2$, see Figure \ref{rpn203}.
\begin{figure}[ht]
\centerline{
\scalebox{.60}{
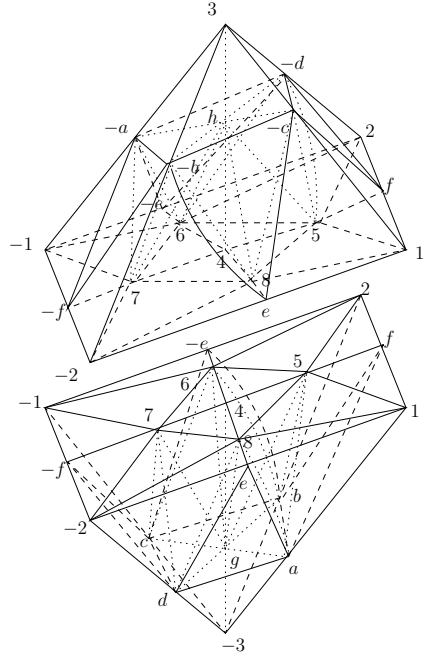
 }
 }
\caption {Triangulations of $B_1$ and $B_2$. }
\label{rpn203}
\end{figure}
By using the triangulations of $B_1$ and $B_2$, we get a triangulation of $\mathbb{RP}^3$, say $K$ with $16$ vertices. Followings are the $3$-simplices of $K$.
\noindent$$158a, 15af, 18ae, 167c, 16ce, 17cf, 158c, 15cf, 18ce, 167a, 17af, 16ae,
256b, 25bf, 26be, 278d,$$ 
$$27df, 28de, 256d, 25df, 26de, 278b, 27bf, 28be, 3abg, 3bcg, 3cdg, 3adg, 3abh, 3bch, 3cdh, 3adh,$$
$$456g, 467g, 478g, 458g,
 456h, 467h, 478h, 458h, 56bg, 67cg, 78dg, 58ag, 56dh, 67ah, 78bh, 58ch,$$ 
 $$ 5abf, 5abg, 6bce, 6bcg, 7cdf, 8ade, 5cdf, 5cdh, 6ade, 6adh, 7abf, 8bce,
 7cdg, 7abh, 8adg, 8bch.$$

Note that $K$ is a $\ZZ_2^3$-equivariant triangulation of $\mathbb{RP}^3$, where the action is determined by 
\begin{equation}\label{rp31}
\begin{array}{ll} \mu_1(5)=6,~ \mu_1(7)=8, ~\mu_1(e)=f, ~ \mu_1(a)=c, ~\mu_2(5)=8, ~\mu_2(6)=7,\\
 ~\mu_2(e)=f, \mu_2(b)=d, ~ \mu_3(a)=c, \text{and} ~\mu_3(b)=d,
\end{array}
\end{equation}
where $\mu_1, \mu_2, \mu_3$ are the standard generators of $\ZZ_2^3$. The atutomorphism group of $K$ is $\ZZ_8 : (\ZZ_2 \times \ZZ_2)$, where $:$ denote the semi-direct product. The above group has a subgroup isomorphic to $\ZZ_2^3$.

In this triangulation, the $\st(e)$ and $\st(f)$ are described in Figure \ref{rpn7}. We can observe that the map 
$$ x : \st(e,K) \to \st(f,K)$$
is a simplicial isomorphism for any $\mu \in \ZZ_2^3$. 
\begin{figure}[ht]
\centerline{
\scalebox{.50}{
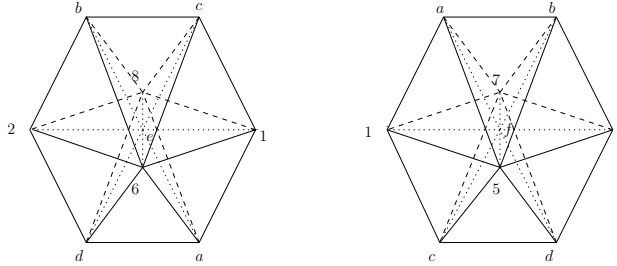
 }
 }
\caption {$\st(e,K)$ and $\st(f,K)$ respectively.}
\label{rpn7}
\end{figure}

Now, we retriangulate $K$ at $\st(e,K)$ with apexes $6,8$ and $\st(f,K)$ with apexes $5,7$. We obtain a new triangulation, say $K'$ of $\RP^3$ with $14$ vertices. The action on $K'$ follows from the action on $K$. Thus, $K'$ is a
a $\ZZ_2^3$-equivariant triangulation of $\RP^3$. The automorphism group of $K'$ is $\ZZ_2\times S_4$ where $S_4$ is the permutation group. 
The $3$-simplices of $K'$ are: 
$$158a, 157a, 168a, 167c, 168c, 158c, 157c, 167a, 256b, 257b, 268b, 278d,
257d, 268d, 256d,$$ 
$$278b, 3abg, 3bcg, 3cdg, 3adg, 3abh, 3bch, 3cdh, 3adh,
456g, 467g, 478g, 458g, 456h, 467h,$$ 
$$ 478h, 458h, 56bg, 67cg, 78dg, 58ag,
57ab, 57cd, 56dh, 68ad, 68bc, 67ah, 78bh, 58ch, 5abg,$$ 
$$5cdh, 6bcg, 6adh, 7cdg, 7abh, 8adg, 8bch.$$

In this $14$-vertex triangulation, the star of vertices $1, 2$ and $3$ are given in Figure \ref{rpn9}.
\begin{figure}[ht]
\centerline{
\scalebox{.50}{
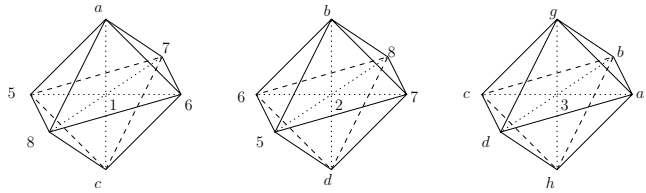
 }
 }
\caption {$\st(1)$, $\st(2)$ and $\st(3)$ respectively.}
\label{rpn9}
\end{figure}
In this case, one can observe that the map
$\mu: \st(i) \to \st(i)$ is a simplicial isomorphism for $i \in \{1, 2, 3\}$ and for any $\mu\in \ZZ_2^3$. In other words $1,2$, and $3$ are the fixed points in $K'$. Now, we retriangulate $K'$ at $\st(i,K)$ with apexes accordingly for $i\in\{1, 2, 3\}$.  We get a new triangulation, say $K''$ of $\RP^3$. The group action on $K''$ directly follows from the group action on $K'$. The automorphism group of $K''$ is $\ZZ_2\times S_4$. 

The $3$-simplices of $K''$ are: 
$$57ac, 68ac, 67ac, 58ac, 56bd, 57bd, 68bd, 78bd, abgh, bcgh, cdgh, adgh,
456g, 467g, 478g, 458g, $$ 
$$456h, 467h, 478h, 458h, 56bg, 67cg, 78dg, 58ag, 57ab, 57cd, 56dh, 68ad, 68bc, 67ah, 78bh, 58ch,$$
$$5abg, 5cdh, 6bcg, 6adh, 7cdg, 7abh, 8adg, 8bch.$$
So, we construct an $11$-vertex $\ZZ_2^3$-equivariant triangulation of $\mathbb{RP}^3$.

\end{example}

\begin{remark}
Since the minimal triangulation of $\mathbb{RP}^3$ contains $11$ vertices, the $11$ vertex triangulation of $\mathbb{RP}^3$ constructed in Example \ref{rp3e} is  the minimal $\ZZ_2^3$-equivariant triangulation. 
\end{remark}

\begin{theorem}\label{thm3}
There is a unique $\ZZ_2^3$-equivariant triangulation of $\mathbb{RP}^3$ with $11$ vertices.
\end{theorem}
\begin{proof}
Let $X$ be a $\ZZ_2^3$-equivariant triangulation of $\mathbb{RP}^3$
with $11$ vertices. Let $\pi : |X| \to \Delta^3$ be the orbit map of the $\ZZ_2^3$-action on $|X|$,
where $\Delta^3$ is the standard $3$-simplex with vertices $e_0, \ldots, e_3$ and codimension one face $F_i$
not containing the vertex $e_i$ for $i=0, \ldots, 3$. Let $X_i = \pi^{-1}(F_i)$ for $i=0, \ldots, 3$.

 Let $$D_{i}=\{\sigma \in X : e_i \in |\sigma|\} ~ \mbox{for} ~ i=0, \ldots, 3.$$
Thus $D_{i}$ is $\ZZ_2^3$-equivariantly homeomorphic to $\mathbb{D}^3$ and the boundary $\partial{D_i}$ is a
subcomplex of $X$ for $i= 0, \ldots, 3$, and a $\ZZ_2^3$-equivariant triangulation of $\Sp^2$. So by Lemma \ref{lem3}, $\mathrm{Card}(V(\partial{D_{i}}))$ is even and
$\mathrm{Card}(V(\partial{D_{i}})) \geq 6$ for $i=0, \ldots, 3$. On the other hand the number of vertices in
$S := V(X)\setminus\{e_0, \ldots, e_3\}$ is even by Lemma \ref{lem1}.
Thus $X$ contains at least one fixed vertex and $X$ can not contain $2$ or $4$ fixed vertices.
  Now, we have the following claim.
\begin{claim}\label{claim}
$X$ does not contain three fixed vertices.
\end{claim}

Suppose that $X$ contains three fixed vertices
namely $e_0, e_1, e_2$. So $\mathrm{Card}(S) = 8$. Let $S = \{x_1, \ldots, x_8\}$. 
Let $A = |X|\setminus \cup_{i=0}^3 X_i$ and $x_i \in A$ for some $i \in \{1, \ldots, 8\}$.
Thus $S = \{\mu x_i : \mu\in \ZZ_2^3\}$. Then $V(X) \cap \pi^{-1}(F\setminus V(F))$ is empty for any proper face $F$ containing the edge $e_0e_3$ of $\Delta^3$. Which contradicts Lemma \ref{lem5}. Therefore, $V(X) \cap A$ is empty and $\pi(x_i) \in \cup_{j=0}^3 F_j\setminus \{e_0, \ldots, e_3\}~ \mbox{for} ~ i=1, \ldots, 8.$

If $\mathrm{Card}(V(\partial{D_0})) = 10$, then $e_1,e_2\in V(\partial{D_0})$, which contradicts Lemma \ref{lem4}. Thus, $6\leq \mathrm{Card}(V(\partial{D_0}))\leq 8$. 
For any $x\in V(\partial{D_0})$, $\pi(x)$ lie either in the interior of an edge $e_0e_i$, or a face $e_0e_ie_j$, where $i,j\in\{1,2,3\}$. Now, from Remark \ref{remark}, we know that if for any $x\in V(\partial{D_0})$,  $\pi(x)$ lie in the interor of an edge $e_0e_i$ for some $i\in\{1,2,3\}$ in $\Delta^3$ then this contributes two vertices to the $X_k$ and $X_l$ where,   $k,l\in\{1,2,3\}\setminus\{i\}$. If $x\in V(\partial{D_0})$ and $\pi(x)$ lie in the interor of a face $e_0e_ie_j$ for some $i,j\in\{1,2,3\}$ in $\Delta^3$ then this contributes four vertices to the $X_l$ where $l\in\{1,2,3\}\setminus\{i,j\}$. 

With the above observation, in either of the cases:  $\mathrm{Card}(V(\partial{D_0}))=6$ or $\mathrm{Card}(V(\partial{D_0}))=8$, we find that there exists an $X_j$, containing at least $6$ vertices. This implies that $\mathrm{Card}(V(\partial{D_{j}})) \leq 5$.
This is a contradiction. This proves the 
 Claim \ref{claim}.

Hence $X$ contains exactly one fixed vertex, say, $e_i$ for some $i\in \{0,1,2,3\}$. So $\mathrm{Card}(S) = 10$. Let $S = \{y_1, \ldots, y_{10}\}$.
Let $y_j \in A$ for some $j \in \{1, \ldots, 10\}$. Then $\{\mu y_j : \mu\in \ZZ_2^3\} \subset S ~ \mbox{and} ~
\mathrm{Card}((S- \{\mu y_j : \mu\in \ZZ_2^3\})) = 2.$ Therefore, $\pi(S\setminus \{\mu y_j : \mu\in \ZZ_2^3\})$ is an interior point
of an edge, say, $e_ke_l$, of $\Delta^3$. So $V(X) \cap \pi^{-1}(F\setminus V(F))$ is empty for any proper
face $F$ containing the edge $e_pe_q$ of $\Delta^3$ where, $\mathrm{Card}(\{e_k,e_l\}\cap \{e_p,e_q\})\leq 1$. This contradicts Lemma \ref{lem5}.
Thus $V(X) \cap A$ is empty and $\pi(y_j) \in \cup_{k=0}^3 F_k \setminus\{e_0, \ldots, e_3\}~ \mbox{for} ~ j=1, \ldots, 10.$

Assume that $\mathrm{Card}( V(\partial{D_i})) \geq 8$. Then again there exists an $X_j$, which contains at least $6$ vertices for some $j \in \{0, \ldots, 3\}\setminus \{i\}$. This implies that $\mathrm{Card}(V(\partial{D_{j}})) \leq 4$. This is a contradiction. Therefore, $\mathrm{Card}( V(\partial{D_i}))=6$. 

From the above discussions, we find that $X$ has exactly one fixed vertex whose degree is six in $X$. Thus, there are at least four vertices in $X$ with degrees $\leq 9$. This implies that $f_1(X)\leq 51$. In \cite{LS}, it was proved that there exists a unique triangulation of $\mathbb{RP}^3$ with $(f_0,f_1)=(11,51)$ and there is no triangulation of $\mathbb{RP}^3$ with $f_1< 51$.
Hence, we deduce that there is a unique $\ZZ_2^3$-equivariant triangulation of $\mathbb{RP}^3$ with $11$ vertices.
\end{proof}


 \section{Equivariant Triangulation of Connected Sums}\label{sec:connectedsum}

In this section, we give a construction method to construct equivariant triangulations of connected sums of two manifolds. 
Also, we  give an equivariant triangulation of $\mathbb{RP}^3\#\mathbb{RP}^3$ on $17$ vertices.  The minimum triangulation of $\mathbb{RP}^3\#\mathbb{RP}^3$ is known on $15$ vertices.

A {\em subcomplex} of a simplicial complex is a subcollection of simplices that also forms a simplicial complex. If $K$ is a simplicial complex and $S\subseteq V(K)$ is a subset then $\mathrm{Ind}_{K}(S)=\{\sigma\in K: V(\sigma)\subseteq S\}$ is called the {\em induced subcomplex} of $K$ induced on the vertex set $S$.

\begin{defn}
 {\rm  Let $K$, $L$ be the triangulation of $n$-manifolds $M$ and $N$, respectively. Let $u$ and $v$ be two vertices in $K$ and $L$, respectively, such that $\lk(u,K)$ and $\lk(v,L)$ are combinatorially isomorphic and $
 \mathrm{Ind}_K(V(\lk(u,K)))=\lk(u,K)$. 
      Suppose that $\psi: \lk(u,K)\rightarrow \lk(v,L)$ is an isomorphism. Define ${\rm ast}(u,K)=\{\sigma\in K: u\notin \sigma\}$ and ${\rm ast}(v,L)=\{\gamma\in L: v\notin\gamma\}$ as the antistar of $u$ and $v$ in $K$ and $L$ respectively. Now, we define the generalized star-connected sum $K\#_{\psi}L$ as the simplicial complex obtained after gluing the faces of ${\rm ast}(u,K)$ with faces of  ${\rm ast}(v,L)$ according to the map $\psi$.} 
\end{defn}
 This generalized star-connected sum is the generalized version of the star-connected sum given in [Definition $1.2$, \cite{KNS}]. Observe that $K\#_{\psi}L$ is a triangulation of $M\#N$.

 \begin{lemma}\label{g2sum}
Let $K,L$ be two triangulation of an  $n$-manifold $M$. Let $\lk(u,K)$ and $\lk(v,L)$ are combinatorially isomorphic and $\mathrm{Ind}_K(V(\lk(u,K)))=\lk(u,K)$.  Suppose that $\psi: \lk(u,K)\rightarrow \lk(v,L)$ is the isomorphism. Then $g_2(K\#_{\psi}L)=g_2(K)+g_2(L)-f_1(\lk(u,K)) + {(n-1)f_0(\lk(u,K))-\frac{(n+1)(n-2)}{2}}.$
 \end{lemma}

\begin{proof}
From the definition of $K\#_{\psi}L$, we have the following.
\begin{itemize}
    \item $f_0(K\#_{\psi}L)=f_0(K)+f_0(L)-f_0(\lk(u,K))-2$.
    \item $f_1(K\#_{\psi}L)=f_1(K)+f_1(L)-2f_0(\lk(u,K))-f_1(\lk(u,K))$.
    \end{itemize}

    Thus,
    \begin{align*}
   g_2(K\#_{\psi}L)=&f_1(K\#_{\psi}L)-(n+1)f_0(K\#_{\psi}L)+\binom{n+2}{2}\\
 =& f_1(K)+f_1(L)-2f_0(\lk(u,K))-f_1(\lk(u,K))\\
 &-(n+1)(f_0(K)+f_0(L)-f_0(\lk(u,K)-2)+\binom{n+2}{2}\\
 =& f_1(K)-(n+1)f_0(K)+\binom{n+2}{2} +f_1(L)-(n+1)f_0(L)+\binom{n+2}{2}\\
 &+(n-1)f_0(\lk(u,K))-f_1(\lk(u,K))+2(n+1)- \binom{n+2}{2}\\
 =& g_2(K)+g_2(L)-f_1(\lk(u,K))+(n-1)f_0(\lk(u,K))-\frac{(n+1)(n-2)}{2}.
    \end{align*}
\end{proof}

\begin{corollary}
    Let $K,L$ be two triangulation of an  $3$-manifold $M$. Let $\lk(u,K)$ and $\lk(v,L)$ are combinatorially isomorphic and $\mathrm{Ind}_K(V(\lk(u,K)))=\lk(u,K)$.  Suppose that $\psi: \lk(u,K)\rightarrow \lk(v,L)$ is the isomorphism. Then $g_2(K\#_{\psi}L)=g_2(K)+g_2(L) -f_0(\lk(u,K))+4.$
\end{corollary}

\begin{proof}
    Since  $\lk(u.K)$ is a triangulated $2$-sphere, $f_1(\lk(u,K))=3f_0(\lk(u,K))-6$.
    Now, the result follows from Lemma \ref{g2sum}.
\end{proof}

\begin{theorem}\label{three}
    The manifold $\RP^3\#\RP^3$ has (at least) three different minimal $g$-vectors.
\end{theorem}
\begin{proof}
In [Theorem 28, \cite{LS}], author proved that there are at least two minimal $g$-vectors for $\RP^3\#\RP^3$: $g=(13,34)$ and $g=(10,36)$. Here, we provide a triangulation of $\RP^3\#\RP^3$ which has $g$-vector $g=(12,35)$.

Consider the $14$ vertex triangulation, $K'$ obtained in Example \ref{rp3e}. We Observe that $\mathrm{Ind}_K(V(\lk(3),K'))=\lk(3,K')$. Now, consider the minimal equivariant triangulation, $K''$ of $\RP^3$ with $11$ vertices given in Example \ref{rp3e}. Then  $\lk(4,K'')$ is isomorphic with $\lk(3,K')$. 

Now, using generalized star connected sum, we construct a triangulation $K'\#K''$ by identifying the $\mathrm{ast}(3,K')$ and $\mathrm{ast}(4,K'')$ according to the isomorphism between them.  $K'\#K''$ is a triangulation of $\RP^3\#\RP^3$ with $17$ vertices, and its $f$-vector is $(17,93,150,75)$. The $g$-vector of $K'\#K''$ is $g=(12,35)$.
\end{proof}

\begin{theorem}\label{connectedsum}
    Let $K$, $L$ be the $G$-equivariant triangulations of $n$-manifolds $M$ and $N$, respectively for some finite group $G$. Let $u\in K$ and $v\in L$ be the fixed vertices under the group action. Let $K$ and $L$ satisfy the following:
\begin{itemize}
    \item There is a combinatorial isomorphism $\psi:\lk(u,K)\rightarrow \lk(v,L)$ such that $\psi(\mu(x))=\mu(\psi(x))$ for all $\mu\in G$ and $x\in \lk(u,K)$,
    
    \item $\mathrm{Ind}_K(V(\lk(u,K)))=\lk(u,K)$, and  $\mathrm{Ind}_L(V(\lk(v,L)))=\lk(v,L).$  
\end{itemize}
Then $K\#_{\psi}L$ is a $G$-equivariant triangulation of $M\#N$.
  
\end{theorem}  
\begin{proof}
We observe that $K\#_{\psi}L$ is a triangulation of the connected sum $M\#N$. 
It remains to show that this triangulation is $G$-equivariant. 
Let $\mu\in G$ be arbitrary.

Let $\sigma$ be a $k$-simplex of $K\#_{\psi}L$ for some $k\leq n$. 
Let $S\subseteq V(K\#_{\psi}L)$ denote the set of vertices that arise from the identification of vertices in $\lk(u,K)$ and $\lk(v,L)$; that is, for each $z\in S$, there exist $x\in \lk(u,K)$ and $y\in \lk(v,L)$ such that $z$ is the image of the identification of $x$ and $y$.

\noindent
\textbf{Case 1.} Suppose that $V(\sigma)\cap S=\emptyset$.  
Then either $\sigma\subseteq K\setminus \st(u,K)$ or $\sigma\subseteq L\setminus \st(v,L)$.  
Since both $K$ and $L$ are $G$-equivariant, $\mu(\sigma)$ is also a $k$-simplex in $K\#_{\psi}L$.

\noindent
\textbf{Case 2.} Suppose that $V(\sigma)\cap S\neq \emptyset$. If $V(\sigma)\subset S$ then $\sigma$ is an identification of a simplex, say $\sigma'$ from $\lk(u,K)$ and a simplex, say $\sigma''$ from $\lk(v,L)$. Thus, $\mu(\sigma)$ is a simplex in $K\#_{\psi}L$ coming from identifying the simplices $\mu(\sigma')$ from $\lk(u,K)$ and $\mu(\sigma'')$ from $\lk(v,L)$. 

Now, consider that $V(\sigma)\not\subset S$. 
Then $\sigma=\sigma'\star\sigma''$ where $\sigma'$ and $\sigma''$ are two simplices and  $V(\sigma)\cap S=V(\sigma'')$.

Then corresponding to $V(\sigma'')$, there exist subsets 
$S_1\subseteq V(\lk(u,K))$ and $S_2\subseteq V(\lk(v,L))$ such that the vertices of $S_1$ are identified with those of $S_2$ via $\psi$.   
Let $\tau=\mathrm{Ind}_K(S_1)$ and $\tau'=\mathrm{Ind}_L(S_2)$ be the simplices in $\lk(u,K)$ and $\lk(v,L)$, respectively.  
The identification $\tau \mathrel{\overset{\scriptscriptstyle \psi}{\sim}} \tau'$ forms $\sigma''$ in $K\#_{\psi}L$.
Since $\lk(u,K)$ and $\lk(v,L)$ are $G$-equivariant, 
$\mu(\tau)$ and $\mu(\tau')$ are simplices in $\lk(u,K)$ and $\lk(v,L)$, respectively, and their identification gives the simplex 
$\mu(\sigma'')$ in $K\#_{\psi}L$.

On the other hand, either $\sigma' \subseteq K\setminus \st(u,K)$ 
or $\sigma'\subseteq L\setminus \st(v,L)$.  
Hence, $\mu(\sigma')$ is a simplex in $K\#_{\psi}L$.  
Therefore, the join 
$\mu(\sigma') \star \mu(\sigma'')$ 
is a $k$-simplex of $K\#_{\psi}L$.

Thus, in both cases, $\mu(\sigma)$ is a simplex of $K\#_{\psi}L$.  
Hence $K\#_{\psi}L$ is $G$-equivariant.
 \end{proof}

\begin{example}
We consider the $14$-vertex triangulation $K'$ of $\RP^3$ constructed in Example \ref{rp3e}. Then, we have  
$$
V(K')=\{1,2,3,\ldots,8\}\cup\{a,b,c,d,g,h\}.
$$  
Observe that $\mathrm{Ind}_K(V(\lk(3,K)))=\lk(3,K)$. Now, take another copy of $K'$, say $L'$, with vertex set  
$$
V(L')=\{1',2',3',\ldots,8'\}\cup\{a',b',c',d',g',h'\}.$$  
Similarly, we have $\mathrm{Ind}_{K'}(V(\lk(3',K')))=\lk(3',K')$. Both triangulations $K'$ and $L'$ are $\ZZ_2^3$-equivariant triangulations of $\RP^3$.  

There exists a combinatorial isomorphism $\psi:\lk(3,K)\to \lk(3',L')$ defined by $\psi(x)=x'$ for each $x\in \lk(3,K)$. Moreover, $\mu(\psi(x))=\psi(\mu(x))$ for all $x\in \lk(3,K)$. Thus, by Theorem \ref{connectedsum}, the connected sum $K'\#_{\psi}L'$ is a $\ZZ_2^3$-equivariant triangulation of $\RP^3\#\RP^3$ with $20$ vertices. Further, we observe that $1,2,4$ are fixed points in $K'\#_{\psi}L'$ and $\lk(1,K'\#_{\psi}L')=B_{\{5,8,7,6\}}(a;c)$,  $\lk(2,K'\#_{\psi}L')=B_{\{5,7,8,6\}}(b;d)$, and $\lk(4,K'\#_{\psi}L')=B_{\{a,b,c,d\}}(g;h)$. Thus, we can remove these fixed points and retriangulate $K'\#_{\psi}L'$ to obtain a new $\ZZ_2^3$-equivariant triangulation of $\RP^3\#\RP^3$ with $17$ vertices.  This triangulation is same as in  Theorem \ref{three}. 
   
\end{example}

\begin{remark}\label{Lutz}
  Let $N$ be a triangulation of $\mathbb{RP}^3$. Then, from \cite{LS}, we have the following lower bounds:
$g_2(N)\geq 17$ if $f_0(N)=11$,
$g_2(N)\geq 22$ if $f_0(N)=12$,
$g_2(N)\geq 21$ if $f_0(N)=13$,
$g_2(N)\geq 23$ if $f_0(N)=14$,
$g_2(N)\geq 22$ if $f_0(N)=15$,
$g_2(N)\geq 19$ if $f_0(N)=16$, and
$g_2(N)\geq 17$ if $f_0(N)=17$.
  \end{remark}

Let $M$ be a $3$-manifold. Let $(f_0(M);d(M))$ denote an admissible pair for $M$, meaning that there exists a triangulation $K$ of $M$ with $f_0(M)$ vertices and a vertex $u \in K$ such that $f_0(\lk(u,K))=d(M)$ and $\mathrm{Ind}_K(V(\lk(u,K))) = \lk(u,K)$.
\begin{lemma}\label{admissible}
    $(f_0(\RP^3);d(\RP^3))\notin \{(11;6),(11;8),(11;10), (12;6),(12;8),(12;10),(13;8),$ $(13;10),(13;12),(14;10),(14;12),(15;12),(15;14),(16;14),(17;16)\}.$
\end{lemma}

\begin{proof}
Let $(f_0;d)$ be an admissible pair for $\mathbb{RP}^3$. Then there exists a triangulation $K$ of $\mathbb{RP}^3$ such that $f_0(K)=f_0$, and there exists a vertex $u\in K$ with $f_0(\lk(u,K))=d$ and $\mathrm{Ind}_K(V(\lk(u,K)))=\lk(u,K)$. Since, apart from the $f_1(\lk(u,K))$ edges, there is no edge in $\mathrm{Ind}_K(V(\lk(u,K)))$, there are at least $\binom{d}{2}-f_1(\lk(u,K))=\binom{d}{2}-3f_0(\lk(u,K))+6$ missing edges in $K$.

Additionally, there are $f_0-d-1$ further missing edges. Therefore,
$f_1(K)\leq \binom{f_0}{2}-\big(\binom{d}{2}-3f_0(\lk(u,K))+6\big)-(f_0-d-1)$.
This implies that
$f_1(K)\leq \frac{f_0(f_0-3)}{2}-\frac{d(d-9)}{2}-5$.

Therefore,
$g_2(K)=f_1(K)-4f_0(K)+10\leq \frac{f_0(f_0-3)}{2}-\frac{d(d-9)}{2}-5-4f_0+10$.
This further simplifies to
$g_2(K)\leq \frac{f_0(f_0-11)}{2}-\frac{d(d-9)}{2}+5$.

We find that $g_2(K)<17$ for all admissible cases of $\mathbb{RP}^3$ from $\{(11;6),(11;8),(11;10),(12;8),$ $(12;10),(13;10),(13;12),(14;10),(14;12),(15;14),(16;14),(17;16)\}$, which contradicts Remark \ref{Lutz}. Further, in the cases $(12,6)$ and $(15,12)$, we obtain $g_2\leq 17$ in both instances, which again contradicts Remark \ref{Lutz}.

The only remaining case is $(13;8)$. In this case, $g_2\leq 22$, and for any vertex $x\in \lk(u,K)$, we have $\lk(x)\cap\lk(u)-\lk(ux)=\emptyset$. Thus, we can contract the edge $ux\in K$. Choose a vertex $y\in \lk(u,K)$ whose degree is at least $4$. Then the $g_2$ of the new triangulation of $\mathbb{RP}^3$ obtained by contracting the edge $uy\in K$ decreases by at least one and is less than or equal to $21$. This implies that there exists a triangulation of $\mathbb{RP}^3$ with $12$ vertices and $g_2\leq 21$, which contradicts Remark \ref{Lutz}.  
\end{proof}

\begin{theorem}
There does not exist a $\mathbb{Z}_2^3$-equivariant triangulation of $\mathbb{RP}^3\#\mathbb{RP}^3$ on at most $16$ vertices that can be obtained, via Theorem \ref{connectedsum}, from two $\mathbb{Z}_2^3$-equivariant triangulations of $\mathbb{RP}^3$.
\end{theorem}

\begin{proof}

Let $K$ and $L$ be two $\mathbb{Z}_2^3$-equivariant triangulations of $\mathbb{RP}^3$ with fixed vertices $u$ and $v$, respectively. Suppose that $\mathrm{Ind}_K(V(\lk(u,K)))=\lk(u,K)$ and $\mathrm{Ind}_L(V(\lk(v,L)))=\lk(v,L)$, and that there exists an isomorphism $\psi:\lk(u,K)\to\lk(v,L)$.

\textbf{Case (i):}
Assume that $K\#_{\psi}L$ is a $\ZZ_2^3$-equivariant triangulation of $\mathbb{RP}^3\#\mathbb{RP}^3$ on $15$ vertices.

Since $\lk(u)$ is a $\mathbb{Z}_2^3$-equivariant triangulation of the $2$-sphere, $d(u)$ must be even and $d(u)=d(v)$. If $d(u)\geq 16$, then $f_0(K),f_0(L)\geq 17$. This would imply
$f_0(K\#_{\psi}L)=f_0(K)+f_0(L)-d(u)-2\geq 16$,
a contradiction. Thus, the possible values of $d(u)$ are $6,8,10,12,$ or $14$.

For $d(u)\in{6,8,10,12,14}$, let $(f_0(K),f_0(L))_{d(u)}$ denote the pair corresponding to the degree $d(u)=d(v)$. Since
$f_0(K)+f_0(L)-d(u)-2=f_0(K\#_{\psi}L)=15$,
an analysis of each value of $d(u)$ yields the following possibilities:
$(11,12)_6,(12,13)_8,(11,14)_8,(11,16)_{10},(12,15)_{10},$ $(13,14)_{10},(13,16)_{12},(14,15)_{12},(15,16)_{14}$.
These possibilities imply the admissibility of the pairs
$(f_0(\mathbb{RP}^3);d(\mathbb{RP}^3))\in\{(11;6),(11;8),(11;10),(12;8),(12;10),(13;10),$ $(13;12),$ $(14;12),(15;14)\}$,
which contradicts Lemma \ref{admissible}.

\textbf{Case (ii):}
Suppose that $K\#_{\psi}L$ is a $\mathbb{Z}_2^3$-equivariant triangulation of $\mathbb{RP}^3\#\mathbb{RP}^3$ on $16$ vertices.

If $d(u)\geq 17$, then $f_0(K),f_0(L)\geq 18$. This would imply
$f_0(K\#_{\psi}L)=f_0(K)+f_0(L)-d(u)-2\geq 17$,
a contradiction. Thus, the possible values of $d(u)$ are $6,8,10,12,14,$ or $16$.

Since
$f_0(K)+f_0(L)-d(u)-2=f_0(K\#_{\psi}L)=16$,
an analysis of each value of $d(u)$ yields the following possibilities:
$(11,13)_6,(12,12)_6,(11,15)_8,(12,14)_8,(13,13)_8,(11,17)_{10},(12,16)_{10},$ $(13,15)_{10},(14,14)_{10},(13,17)_{12},(14,16)_{12},(15,15)_{12},(15,17)_{14},(16,16)_{14},(17,17)_{16}$.
These possibilities imply the admissibility of the pairs
$(f_0(\mathbb{RP}^3);d(\mathbb{RP}^3))\in\{(11;6),(11;8),(11;10),$ $(12;6),(12;8),(12;10),(13;8),(13;10),(13;12),(14;10),(14;12),(15;12),(15;14),(16;14),$ \\ $(17;16)\}$,
which contradicts Lemma \ref{admissible}.
\end{proof}

\begin{question}
    Is their a $\ZZ_2^3$-equivariant triangulation of $\RP^3\#\RP^3$ with $15$  or $16$ vertices?
\end{question}

\section{Equivariant Triangulation on $\mathbb{RP}^4$}\label{sec:RP4}
In this section, we study equivariant triangulations of $\RP^4$ with at most $17$ vertices. We note that any triangulation of $\RP^4$ contains at least 16 vertices.

\begin{lemma}\label{lem:nointerior}
   Let $X$ be a $\ZZ_2^4$-equivariant triangulation of $\RP^4$ with less than 18 vertices. Let $\pi:|X|\rightarrow \D^4$ be the orbit map. There is no face of dimension three or more in $\D^4$ whose interior conatins a point $x$ such that $\pi^{-1}(x) \cap V(X) \neq \emptyset$. 
\end{lemma}

\begin{proof}
Let the vertex set of $\D^4$ be $\{e_0, e_1, e_2, e_3, e_4\}$. Suppose that the interior of $\D^4$ contains a point $x$ such that $\pi^{-1}(x)$ contains a vertex, say $x'$ in $X$. Then all four coordinates of $x'$ must be non-zero. Since $\ZZ_2^4$ acts on $x'\in V(X)$, there are at least $16$ vertices in $X$ with all coordinates non-zero.

Define $D_i = \{\sigma \in X : e_i \in |\sigma|\}$ for $0 \leq i \leq 4$. Then each $D_i$ is a $4$-dimensional ball, and $\partial D_i$ is a subcomplex of $X$ which forms a $\ZZ_2^4$-equivariant triangulation of the $3$-sphere. 
Since $X$ contains at most $17$ vertices, $\partial D_i$ must contain $\mu(x')$ for some $\mu\in \ZZ_2^4$. This implies that $\partial D_i$ contains a vertex with all four coourdinates being non-zero. Since $\partial D_i$ is a $\ZZ_2^4$-equivariant triangulation of a $3$-sphere, $\partial D_i$ must contain at least $16$ vertices with all the coordinates non-zero. On the other hand $\partial D_i$ contains even number of vertices, $\mathrm{Card}(V(\partial D_i))=16$. From Lemma~\ref{lem6}, this is not possible.

Now, we consider the interior of a $3$-dimensional face of $\D^4$, say $\tau$, which contains a point $y$ such that $\pi^{-1}(y)$ contains a vertex in $X$. Without loss of generality, assume that $V(\tau) = \{e_0, e_1, e_2, e_3\}$. Let $z \in \pi^{-1}(y)$ be a vertex in $X$. Then three coordinates of $z$ must be non-zero. Since $\ZZ_2^4$ acts on the vertices of $X$, there are at least $8$ vertices in $X$ with exactly three non-zero coordinates.

If $V(\partial D_j) = \pi^{-1}(y)$ for some $0 \leq j \leq 3$, then $\partial D_j$ is a $\ZZ_2^3$-equivariant triangulation on $8$ vertices whose every vertex has three non-zero coordinates. This contradicts Corollary~\ref{8vertex}. Thus, $\mathrm{Card}(V(\partial D_j)) \geq 10$ for each $0 \leq j \leq 3$. Since $X$ contains at most $17$ vertices, $\mathrm{Card}(V(\partial D_4)) \geq 8$, and from Lemma~\ref{lem4}, $V(\partial D_4) \cap \pi^{-1}(\tau) = \emptyset$, we have $\mathrm{Card}(V(\partial D_4)) = 8$. If $\partial D_4$ is of type Corollary~\ref{8vertex}~$(II)$ or $(III)$, then there exists a $k \in \{0,1,2,3\}$ such that $\mathrm{Card}(V(\partial D_k)) = 10$. By Lemma~\ref{missingedges}, $\partial D_k$ contains at least $16$ missing edges. Hence, $f_1(\partial D_k) \leq 29$. Since $\partial D_k $ is a triangulated $3$-sphere on $10$ vertices, we have $g_2(\partial D_k) \geq 0$, i.e., $4f_0(\partial D_k) - 10 \leq f_1(\partial D_k)$. It follows that $30 \leq f_1(\partial D_k)$, a contradiction.

Thus, $\partial D_4$ must be of type Corollary~\ref{8vertex}~$(I)$. This implies that $\pi(V(\partial D_4)) = \{a, b\}$, where $a$ and $b$ lie in the interiors of two complementary triangles in $\D^4$ containing the vertex $e_4$. Without loss of generality, assume that $a$ lies in the interior of $e_0e_1e_4$ and $b$ belongs to the interior of $e_2e_3e_4$, such that $\pi^{-1}(a)$ and $\pi^{-1}(b)$ contribute four vertices each to $V(\partial D_4)$.
Since $\mathrm{Card}(V(\partial D_j)) \geq 10$ for each $0 \leq j \leq 3$, we conclude that $\pi^{-1}(a)$ contributes four vertices to $\partial D_0$ and the same four vertices to $\partial D_1$. Similarly, $\pi^{-1}(b)$ contributes four vertices to $\partial D_2$ and $\partial D_3$. Thus, $\mathrm{Card}(V(\partial D_j)) = 12$ for every $0 \leq j \leq 3$. Let $V(\partial D_0) = S \cup T$, where $S$ contains all $8$ vertices with three non-zero coordinates and $T$ contains four vertices with only two non-zero coordinates.

Now, $z \in S$ can be joined only to the vertices of $S$ that differ in exactly one coordinate, by Lemma~\ref{missingedges}. Hence, $z$ can be adjacent to at most three vertices in $S$ and at most four vertices in $T$. Let $p, q, r \in S$ be the possible adjacent vertices of $z$ from $S$. Suppose that $p$ is adjacent to $z$ in $\partial D_0$. Then $\lk(p, \lk(z, \partial D_0))$ is a cycle of length at least three. Since $q$ and $r$ differ from $p$ in two coordinates, they cannot be joined with $p$ in $\partial D_0$. Therefore, $\lk(p, \lk(z, \partial D_0))$ would be a cycle consisting of vertices from $T$. This is not possible, as $T$ contains two pairs of antipodal vertices in $\partial D_0$. Hence, $z$ cannot be joined with $p$, $q$, or $r$ in $\partial D_0$. On the other hand, since $T$ contains two pairs of antipodal vertices in $\partial D_0$, the possibility $V(\lk(z, \partial D_0)) = T$ is also ruled out.
Therefore, $\mathrm{Card}(V(\partial D_0)) = 12$ is not possible, and thus $\partial D_4$ cannot be of type Corollary~\ref{8vertex}~$(I)$. Consequently, there is no $3$-dimensional face of $\D^4$ whose interior contains a point $x$ such that $\pi^{-1}(x)$ contains a vertex in the triangulation $X$.
\end{proof}

\begin{remark}
 The previous lemma describes the vertex distribution of a $\ZZ_2^4$-equivariant triangulation of $\RP^4$ with at most $17$ vertices and suggests that these vertices arise from the interiors of either triangles or edges of $\D^4$.
\end{remark}

\begin{lemma}\label{atmost9}
Let $X$ be a $\ZZ_2^4$-equivariant triangulation of $\RP^4$.  If $f_0(X)\leq 17$ then for any $3$-simplex $\sigma$ in $\D^4$, $\pi^{-1}(\sigma)$ contains at most $9$ vertices. Moreover, for $x\in V(\D^4)\setminus V(\sigma)$, $Card(V(\partial D_x))=8$. 
\end{lemma}

\begin{proof}
   On the contrary, suppose that $\pi^{-1}(\sigma)$ contains $10$ or more vertices in $X$ for some $3$-simplex $\sigma$ in $\D^4$. From Lemma~\ref{lem4}, we get that $V(\pi^{-1}(\sigma)) \cap V(\partial D_{V(\D^4 \setminus \sigma)}) = \emptyset$ and $\mathrm{Card}(V(\partial D_{V(\D^4 \setminus \sigma)})) \geq 8$. Therefore,
$\mathrm{Card}(V(X)) \geq \mathrm{Card}(V(\pi^{-1}(\sigma)) \cup V(\partial D_{V(\D^4 \setminus \sigma)})) \geq 18$, which contradicts the assumption that $\mathrm{Card}(V(X)) \leq 17$. Hence, $\pi^{-1}(\sigma)$ contains $9$ or fewer vertices in $X$ for every $3$-simplex $\sigma$ in $\D^4$.

From Lemma~\ref{lem3}, for $x \in V(\D^4) \setminus V(\sigma)$, $\mathrm{Card}(V(\partial D_x))$ is even and $\mathrm{Card}(V(\partial D_x)) \geq 8$. If $\mathrm{Card}(V(\partial D_x)) \geq 10$, then $\mathrm{Card}(V(\pi^{-1}(\sigma)) \cup V(\partial D_x)) \geq 18$, which again contradicts the assumption that $\mathrm{Card}(V(X)) \leq 17$. Hence, $\mathrm{Card}(V(\partial D_x)) = 8$.
\end{proof}

We adhere the notation from the proof of Lemma \ref{lem:nointerior} in the rest. One can see from the definition of $D_i$ that $\partial D_i$ does not contain fixed vertices for each $i$. Thus,
\[
\bigcup_{i \in \{0,1,2,3,4\}}V( \partial D_i) \subseteq V(X) \setminus \{e_0, e_1, e_2, e_3, e_4\}.
\]
 It follows from Lemma \ref{lem1}, $\mathrm{Card}\left(\bigcup_{i \in \{0,1,2,3,4\}}V( \partial D_i) \right )$ is even.

Further, for each vertex $x \in \partial D_i$ for $0\leq i\leq 4$, $\pi(x)$ lies in the interior of either a triangle containing the vertex $e_i$ or an edge containing the vertex $e_i$ in $\D^4$.  
From Lemma \ref{lem4}, we know that 
$$
V\big(\pi^{-1}(e_i e_j e_k e_l)\big) \cap V\big(\partial D_{V(\D^4) \setminus \{e_i, e_j, e_k, e_l\}}\big) = \emptyset$$
for some distinct $i, j, k, l \in \{0,1,2,3,4\}$.  
Thus, $\mathrm{Card}(V(X))$ is at least the sum of the vertices in $\partial D_{V(\D^4) \setminus \{e_i, e_j, e_k, e_l\}}$ and the vertices arising from the interiors of the edges and triangles of the $3$-simplex $e_i e_j e_k e_l$.

\sloppy
\begin{theorem}\label{thm7}
 There is no $\ZZ_2^4$-equivariant triangulation of $\mathbb{RP}^4$ with at most $17$ vertices.
\end{theorem}
\begin{proof}
Let $X$ be a $\ZZ_2^4$-equivariant triangulation of $\mathbb{RP}^4$ with $16$ vertices.  By Lemma \ref{lem1}, number of vertices in $R =V(X)\setminus \{e_0,\ldots,e_4\}$ is even. First, assume that $V(X)$ does not contain any fixed vertex. For each $i\in \{0,\ldots,4\}$, $V(\partial D_i)\subseteq V(X)$ and  by Lemma \ref{lem3}, $\mathrm{Card}( V(\partial D_i)) \geq 8$. Further, $e_i\notin \partial D_i$ for all $i\in \{0,\ldots,4\}$. This implies that $\bigcup_{i=0}^{4}(V(\partial D_i))=R$.

$\textbf{Case 1}:$ Suppose that $\mathrm{Card}(V(\partial  D_i)) \geq 14$  for some $i \in \{0, \ldots, 4\}$.

If $\mathrm{Card}(V(\partial D_j)) = 8$ for some $j\neq i$, then by Corollary~\ref{8vertex}, $\pi(V(\partial D_j))$ lies in one of the following configurations:  
\begin{itemize}
    \item[(i)] Interiors of four edges of $\D^4$ containing the vertex $v_j$.

    \item [(ii)]  Interiors of two edges containing $v_j$ and one complementary triangle containing $v_j$ in $\Delta^4$.

    \item  [(iii)] Interior of two complementary triangles containing $v_j$ in $\Delta^4$.
\end{itemize}
   
   However, since $\mathrm{Card}(V(\partial D_j) \cap V(\partial D_i)) \ge 6$, at least $6$ vertices of $V(\partial D_j)$ must come from the interiors of some faces in $\D^4$ containing $v_iv_j$. From the above configurations, this is not possible.
Hence, we conclude that $\mathrm{Card}(V(\partial D_j)) \ge 10$.

Now, since $\mathrm{Card}(V(\partial D_j) \cap V(\partial D_i)) \ge 8$ for all $j \neq i$, the pigeonhole principle implies that  
$V(\partial D_i) \cap V(\partial D_k) \cap V(\partial D_l) \cap V(\partial D_m) \neq \emptyset$  
for some distinct $i, k, l, m \in \{0,1,2,3,4\}$. This means that there exists a point $z$ in the interior of a simplex in $\D^4$ containing the vertices $e_i, e_k, e_l, e_m$ such that $\pi^{-1}(z) \subseteq V(\partial D_i)$. This contradicts Lemma~\ref{lem:nointerior}. Hence, there is no $i \in \{0,1,2,3,4\}$ such that $\mathrm{Card}(V(\partial D_i)) \ge 14$.

$\textbf{Case 2}:$ Suppose that $\mathrm{Card}(V(\partial D_i))= 12$ for some $i\in \{0,1,\ldots,4\}$. 

The points of $\pi(V(\partial D_i))$ lie either in the interiors of triangles or in the interiors of edges, by Lemma~\ref{lem:nointerior}, . For each point $x\in \pi(V(\partial D_i))$ in the interior of the triangle, $\pi^{-1}(x)$ consists of four vertices in $\partial D_i$, and for each point $x\in \pi(V(\partial D_i))$ in the interior of an edge, $\pi^{-1}(x)$ consists of two vertices in $\partial D_i$.

If there exists a vertex that belongs to four or more of the sets $V(\partial D_j)$ for $j\in\{0,\ldots,4\}$, then there exists an interior point $x$ in a $3$-simplex $\sigma$ such that $\pi^{-1}(x)$ contains a vertex of $X$. This contradicts Lemma~\ref{lem:nointerior}. Hence, each vertex of $X$, except for the fixed vertex, lies in at most three of the sets $V(\partial D_j)$ for $j\in\{0,\ldots,4\}$.

Thus, the total multiplicity of the vertices is at most $16 \times 3 = 48$. Therefore, $\sum_{j \in \{0,1,2,3,4\}} \mathrm{Card}(V(\partial D_j)) \le 48$. This implies that $\sum_{j \ne i} \mathrm{Card}(V(\partial D_j)) \le 36$. 

Let $j, k, l, m$ denote the remaining four elements of $\{0,1,2,3,4\} \setminus \{i\}$. Since $\mathrm{Card}(V(\partial D_s)) \ge 8$ and is even for each $s \ne i$, the possible tuples $(\mathrm{Card}(V(\partial D_j)), \mathrm{Card}(V(\partial D_k)), \mathrm{Card}(V(\partial D_l)), \mathrm{Card}(V(\partial D_m)))$ are $(8,8,8,8)$, $(8,8,8,10)$, $(8,8,8,12)$, or $(8,8,10,10)$.

Observe that for the first three possibilities, there exists some $s \in \{j, k, l\}$ such that either $V(\partial D_s)$ intersects $V(\partial D_i)$ in more than four vertices, or there exists a common vertex in $V(\partial D_s)$ for $s \in \{j, k, l, m\}$, contradicting Corollary~\ref{8vertex} or Lemma~\ref{lem:nointerior}, respectively. Hence, we are left with the only possible case $(8,8,10,10)$.

It follows that $V(\partial D_j)$ and $V(\partial D_k)$ intersect $V(\partial D_i)$ in exactly four vertices each, and $V(\partial D_j)$ and $V(\partial D_k)$ intersect each other in exactly four vertices from $V(X) \setminus V(\partial D_i)$. This implies that the interior of the triangle $e_j e_k e_s$ in $\D^4$ contributes four vertices to both $V(\partial D_j)$ and $V(\partial D_k)$ for some $s \in \{l, m\}$. 

By Corollary~\ref{8vertex}, for some $t \in \{l, m\} \setminus \{s\}$, the interiors of triangles $e_j e_i e_t$ and $e_k e_i e_t$ contribute four vertices to $V(\partial D_j)$ and $V(\partial D_k)$, respectively. Since $V(\partial D_t)$ contains $10$ vertices, there is a $p\neq t$ such that the interior of $e_pe_t$ contribute at least two vertices to the $V(\partial D_t)$. If $p=i,p=j$, or $p=k$, then the $3$-simplex $e_je_ke_ie_t$ contribute at least $10$ vertices to the $V(X)$. This contradicts Lemma \ref{atmost9}. Thus, the interior of $e_te_s$ must contribute at least two vertices to the $V(\partial D_t)$. Since from Lemma \ref{lem4}, $V(\pi^{-1}(e_je_ke_se_t))\cap V(\partial D_i)=\emptyset $, and $\bigcup_{r=0}^{4}\partial D_r$ contains at most $16$ vertices, $\mathrm{Card}(V(\partial D_i))\leq 10$. This is a contradiction.

$\textbf{Case 3}:$ Suppose that $\mathrm{Card}(V(\partial D_i)) = 10$ for some $i\in \{0,1,\ldots,4\}$. 

If $\mathrm{Card}(V(\partial D_j)) = 10$ for all $j \in \{0,1,2,3,4\}$, then some vertex must belong to at least four of the sets $V(\partial D_j)$, contradicting Lemma~\ref{lem:nointerior}. Hence, there exists $s \ne i$ such that $\mathrm{Card}(V(\partial D_s)) = 8$. Since $\mathrm{Card}(V(\partial D_j)) = 10$, the vertex distribution for $\partial D_j$ is one of the following:
\begin{itemize}
    \item [(i)] Eight vertices comes from the interiors of triangles containing $e_j$ in $\D^4$ and two vertices comes from the interior of an edge containing $e_j$ in $\D^4$.

    \item [(ii)] Four vertices comes from the interior of a triangle containing $e_j$ and the remaining $6$ vertices comes from the interiors of  edges containing $e_j$ in $\D^4$.
\end{itemize}

\textbf{Subcase (i):} Suppose that the  eight vertices of $V(\partial D_i)$ comes from the interiors of triangles in $\D^4$.

Suppose that both the triangles belong to a tetrahedron $\sigma$ with $V(\sigma) = \{e_i, e_j, e_k, e_l\}$.  
If both triangles contributing to $V(\partial D_i)$ are the same, say $e_i e_j e_k$, then the  tetrahedron $e_i e_j e_k e_m$ for some $m \notin \{i,j,k,l\}$ would contribute more than nine vertices to $V(X)$, contradicting Lemma~\ref{atmost9}.  
Hence, two distinct triangles of $\sigma$, say $e_i e_j e_k$ and $e_i e_j e_l$, each contribute four vertices to $V(\partial D_i)$.

The remaining two vertices of $V(\partial D_i)$ must then come from the interior of an edge $e_i e_m$, where $m \notin  \{i,j,k,l\}$.  Since $V(\pi^{-1}(e_ie_je_ke_l))\cap V(\partial D_m)=\emptyset$, by Lemma~\ref{lem3}, $\mathrm{Card}(V(\partial D_m)) = 8$, and the interior of the edge $e_i e_m$ also contributes two vertices to $V(\partial D_m)$.  Thus, $\partial D_m$ must be either of type Corollary \ref{8vertex} $(II)$ or $(III)$. 
If $\partial D_m$ is of type  Corollary~\ref{8vertex} $(III)$, then the interiors of $e_m e_j$, $e_m e_k$, and $e_m e_l$ each contribute two vertices to $V(\partial D_m)$, so the tetrahedron $e_i e_j e_k e_m$ would contribute at least ten vertices to the $V(X)$, a contradiction. Now, suppose that $\partial D_m$ is of type   Corollary \ref{8vertex} $(II)$. Then the following situation arises.
\begin{itemize}
\item Let interior of the edge $e_je_m$ and interior of triangle $e_ke_le_m$ contribute two and four vertices to the $\partial D_m$, respectively. Since $R$ contains at most $16$ vertices and the interiors of faces $e_ie_je_k,e_ie_je_l,e_ke_le_m,e_ie_m,e_je_m$ contribute the $16$ vertices to $X$, $\partial D_k$ as well as $\partial D_l$ contain $8$ vertices and are of type Corollary \ref{8vertex} $(I)$. On the other hand, if $\partial D_j$ contains $8$ vertices then its vertices arise from the interior of faces $e_ie_je_k$ and $e_ie_je_l$  which are not complementary faces containing $e_j$. Thus, $\partial D_j$ contains  $10$ vertices. Now, since the vertices of $\partial D_k$ arise form the interior of $e_ie_je_k$ and $e_ke_le_m$, without loss of generality assume that the vertex set of $\partial D_k$ is $\{(\pm a,\pm b,0,0),(0,0,\pm c,\pm d)$, for some non-zero real numbers $a,b,c,d$. Now, since the same interior point of $e_ke_le_m$ contributes vertices to the $\partial D_l$, $\{(0,0,\pm c,\pm d)\}\subset  V(\partial D_l)$. Thus, the other four vertices of $\partial D_l$ must be of type $\{(\pm e, \pm f,0,0)\}$ for some non-zero $e,f$ and corresponding to the face $e_ie_je_l$. Since the vertices of $\partial D_i$ arise from the interios of faces $e_ie_je_k,e_ie_je_l$, and $e_ie_m$, $\{(\pm a,\pm b,0,0),(\pm e, \pm f,0,0) \}\subset V(\partial D_i)$. The other two vertices of $\partial D_i$ comes from the interior of the edge $e_ie_m$, thus these two vertices must be antipodal and contains exactly one non-zero cordinate. Therefore for any choices of the two vertices arising from $e_ie_m$, all vertices of $V(\partial D _i)$ contains a zero coordinate, which is not possible. 

\item If interior of the edge $e_ke_m$ (resp., $e_le_m$) and interior of triangle $e_je_le_m$ (resp., $e_je_ke_m$) contribute two and four vertices to the $\partial D_m$, respectively, then the $3$-simplex $e_ie_je_le_m$ (resp., $e_ie_je_ke_m)$ contribute at least $10$ vertices to the $V(X)$. This situation is not possible.

\end{itemize}

Now, consider that the two triangles whose interiors contribute $8$ vertices to the $\partial D_i$ are from two different tetrahedrons. Then both triangles must be complementary triangles containing $e_i$. Let the interiors of $e_ie_je_k$ and $e_ie_le_m$ contribute $8$ vertices to the $\partial D_i$. Without loss of generality, assume that the other two vertices of $\partial D_i$ comes from the interior of the edge $e_ie_j$. 
Since the $2$-simplex $e_ie_je_k$ contribute at least $6$ vertices to the $V(X)$, the  interior of $e_ie_le_m$ must contribute four vertices to the $\partial D_l$ and $\partial D_m$ both. If $\partial D_m$ is of type Corollary \ref{8vertex} $(II)$, then the interiors of $e_je_m$, $e_ke_m$ contribute the other four vertices to the $\partial D_m$. Then the $3$-simplex $e_ie_je_ke_m$ contributes at least $10$ vertices to the $V(X)$, which is a contradiction. Now, assume that $\partial D_m$ is of type Corollary \ref{8vertex} $(I)$. Then the other four vertices arise from the interior of $e_je_ke_m$. Then again the $3$-simplex $e_ie_je_ke_m$ contributes at least $10$ vertices to the $X$. This is a contradiction. Hence, this subcase is not possible.

\textbf{Subcase (ii):} Suppose that only one triangle, say $e_i e_j e_k$, contributes four vertices to $V(\partial D_i)$.  
Let $e_l, e_m$ be the remaining vertices of $\mathbb{D}^4$.  
If the remaining six vertices of $V(\partial D_i)$ come from edges or triangles within $e_i e_j e_k e_l$ or $e_i e_j e_k e_m$, then by Lemma~\ref{atmost9}, $V(X)$ would have more than nine vertices, a contradiction.  
Thus, at most two of these vertices can come from edges of $e_ie_je_k$ containing $e_i$.

Assume $e_i e_j$ contributes two such vertices.  
Then by Lemma~\ref{lem3}, edges $e_ie_l$ and $e_ie_m$ also contribute two vertices each to $V(\partial D_i)$. Now, since the $3$-simplex $e_ie_je_ke_l$ contributes at least $8$ vertices to the $X$, $\mathrm{Card}(V(\partial D_m)) = 8$ and the interior of $e_ie_m$ must contribute $2$ vertices to the $\partial D_m$.   
If $\partial D_m$ is of type Corollary \ref{8vertex} $(III)$ , then the interiors of $e_m e_j$, $e_m e_k$, and $e_m e_l$ each contribute two vertices to $V(\partial D_m)$, forcing $e_i e_j e_k e_m$ to contribute at least  ten vertices to $X$, again a contradiction.  
Suppose that $\partial D_m$ is of type Corollary \ref{8vertex} $(II)$. Then for any vertex distribution of $\partial D_m$ in $\D^4$, $e_ie_je_ke_m$ contribute at least $10$ vertices to the $X$, a contradiction.

Finally, if no edge of $e_i e_j e_k$ containing $e_i$ contributes vertices to $V(\partial D_i)$, then the remaining six vertices must come from the interiors of $e_i e_l$ and $e_i e_m$, which 
If the interior of $e_ie_l$ or $e_ie_m$ contribute $6$ verices then either $e_ie_je_ke_l$ or $e_ie_je_ke_m$ contribute at least $10$ vertices to $X$, a contradiction. Thus, we assume that four vertices comes from the interior of $e_ie_l$ and two vertices comes from the $e_ie_m$. Then since the $3$-simplex $e_ie_je_ke_l$ contribute at least  $8$ vertices to the $X$, $\partial D_m$ is of type Corollary \ref{8vertex}$(II)$or $(III)$ . If it is of type Corollary \ref{8vertex} $(III)$ , then $e_ie_je_ke_m$ contribute $10$ vertices to the $X$, a contradiction. Suppose that $\partial D_m$ is of type Corollary \ref{8vertex} $(II)$. Then for any distribution of vertices of $\partial D_m$, either $e_i{e_j}e_{k}e_{m}$, $e_{i}e_{j}e_{l}e_{m}$ or $e_{i}e_{k}e_{l}e_{m}$ contribute at least $10$ vertices to the $X$, a contradiction.

Hence, all possible configurations lead to contradictions.

$\textbf{Case 4}:$ Suppose that $\mathrm{Card}( V(\partial D_i)) = 8$  for $i=0, \ldots, 4$. 

Fix an $i \in \{0, \ldots, 4\}$. Then, there are $3$ possibilities for the triangulation of $\partial D_i$ as described in Corollary \ref{8vertex}. 
If the vertex set of $\partial D_i$ is of type Corollary \ref{8vertex} $(III)$ for some $i\in \{0,1,2,3,4\}$, then $\pi(V(\partial D_i))$ belongs to the interiors of edges of $\Delta^4$ containing the vertex $e_i$, and thus interior of each edge contribute two vertices to the $V(\partial D_i)$. Now, if the vertex set of $\partial D_x$ is of type Corollary \ref{8vertex} $III$ for all $x\in \{0,1,2,3,4\}\setminus \{i\}$
 then interior set  of each edge of $\D^4$ contributes two vertices to the $V(X)$ and then $X$ has at least $20$ vertices. This is a contradiction.
Let for some $j\neq i$, $e_je_le_m$  contributes four vertices to the $\partial D_j$. If $i\in \{l,m\}$ then $e_ie_je_pe_q$ contribute at least $10$ vertices to the $V(X)$ for some $p\in \{l,m\}\setminus \{i\}$ and $q\in \{0,1,2,3,4\}\setminus \{i,j,p\}$. This is a contradiction. If $i\notin\{l,m\}$, then the simplex $e_ie_je_le_m$ contribute at least $10$ vertices to the $V(X)$. This is a contradiction. This implies that for any $x\in \{0,1,2,3,4\}\setminus \{i\}$, $\partial D_x$ cannot be of type Corollary \ref{8vertex} $(I)$ or $(II)$.

Suppose that the vertex set of $\partial D_x$ for all $x\in\{0,\ldots,4\}$ is of type Corollary \ref{8vertex} $(I)$. Let for some $i\in\{0,1,2,3,4\}$, interiors of the  $2$-simplices, say $e_ie_je_k$ and $e_ie_le_m$ in $\D^4$   contributes $8$ vertices to the $V(\partial D_i)$ where $j,k,l,m$ are distinct numbers from $\{0,1,2,3,4\}\setminus \{i\}$. Now, we have the following situations.

\begin{itemize}
    \item The interior of $e_ie_je_k$  does not contribute the same four vertices of $\partial D_i$ to $\partial D_j$. Thus, interior of $e_je_le_m$   also does not contribute any vertex to $V(\partial D_j)$. Let the interiors of $e_je_ke_l$ and $e_je_ie_m$ contribute the $8$ vertices to $\partial D_j$. Then the $16$ vertices of $V(X)$ comes from the interior of faces $e_ie_je_k$, $e_ie_le_m$, $e_je_ke_l$ and $e_ie_je_m$. Then from Corollary \ref{8vertex}, we find that the interiors of $e_ie_le_m$,  and $e_ie_je_m$ cannot contribute the $8$ vertices to the $\partial D_m$ as they are not complementary triangles containing $e_m$. Thus, this case is not possible. Similarly, this case is not possible if the interiors of $e_je_ie_l$ and $e_je_ke_m$ contribute the $8$ vertices to $\partial D_j$.
    
   \item The interior of $e_ie_je_k$ contribute the same four vertices of $\partial D_j$ to $\partial D_j$ and then the interior of $e_je_le_m$ also contribute the four vertices to the $V(\partial D_j)$. Thus, for $V(\partial D_l)$, the contributing triangles must be $e_je_le_m$ and $e_le_ie_k$. Then observe that the $16$ vertices of $V(X)$ comes from the interiors of $e_ie_je_k$, $e_ie_le_m$, $e_je_le_m$, and $e_le_ie_k$. Then we can see that the interiors of $e_ie_le_m$,  and $e_je_le_m$ cannot contribute the $8$ vertices to the $\partial D_m$ as they are not complementary triangles containing $e_m$.    
\end{itemize}

Suppose that the vertex set of $\partial D_i$ for some $i\in\{0,\ldots,4\}$ is of type Corollary \ref{8vertex} $(II)$. Then there is a $2$-simplex, say $e_ie_je_k$ in $\D^4$  whose interiors contributes $4$ vertices to the $V(\partial D_i)$ and two edges, say $e_ie_l$ and $e_ie_m$ whose interiors contributes two vertices each to the  $V(\partial D_i)$ where $l\neq m$ and $l,m\notin\{j,k\}$. Since the vertex set of $\partial D_j$ cannot be of type of Corollary \ref{8vertex} $(III)$, it is of type Corollary \ref{8vertex} $(I)$ or $(II)$. 

\textbf{Subcase (i):} Suppose that the vertex set of $\partial D_j$ is of type Corollary \ref{8vertex} $(II)$, then, we have the following possibilities:
\begin{itemize}
\item The interior of $e_ie_je_k$  does not contribute the same four vertices of $\partial D_i$ to $\partial D_j$. If the interiors of $e_je_ie_l$(resp. $e_je_ke_l)$, $e_je_k$ (resp., $e_je_i$), and $e_je_m$ contributes the $8$ vertices to the $\partial D_j$, then the $3$-simplex $e_ie_je_ke_l$ contributes at least $10$ vertices to the $V(X)$. This is a contradiction. 

If the interiors of $e_je_ie_m$ (resp. $e_je_ke_m)$, $e_je_i$ (resp. $e_je_k$), and $e_je_l$ contributes the $8$ vertices to the $\partial D_j$, then the $3$-simplex $e_ie_je_ke_m$ contributes at least $10$ vertices to the $V(X)$. This is a contradiction.

If the interiors of $e_je_le_m$ , $e_je_i$  and $e_je_k$ contributes the $8$ vertices to the $\partial D_j$, then the $3$-simplex $e_ie_je_ke_l$ contribute at least $10$ vertices to the $V(X)$. This is a contradiction.

    \item The interior of $e_ie_je_k$ contribute the same four vertices of $\partial D_i$ to $\partial D_j$ and then the interior of the edges of $e_je_l$ and $e_je_m$ contribute two vertices each to $\partial D_j$. In this case, the $3$-simplex $e_ie_je_me_l$ contributes at least $8$ vertices to the $V(X)$. Since $V(\pi^{_1}(e_ie_je_le_m))\cap V(\partial D_k)=\emptyset$, the interior of $e_ie_je_k$ must contribute $4$ vertices to the $V(\partial D_k)$. Now, the other four vertices of $V(\partial D_k)$ either comes from the interior of triangle $e_ke_le_m$ or the interior of edges $e_ke_l$ and $e_ke_m$. In the later case, the $3$-simplex $e_ie_je_ke_l$ contributes at least $10$ vertices to the $V(X)$. This is a contradiction. Thus, the interiors of  $e_ke_le_m$ and $e_ie_je_k$ contributes the $8$ vertices to the $V(\partial D_k)$. 
    Now, the $16$ vertices of $X$ comes from the interiors of faces $e_ie_je_k, e_ie_l,e_ie_m, e_je_l,e_je_m, e_ke_le_m$. Therefore, the interiors of faces  $e_ie_m, e_je_m, e_ke_le_m$ must contribute the $8$ vertices to the $\partial D_m$. Further, the interiors of faces  $e_ie_l, e_je_l, e_ke_le_m$  contribute the $8$ vertices to the $\partial D_l$. 
\end{itemize}

\textbf{Subcase (ii):} Suppose that the vertex set of $\partial D_j$ is of type Corollary \ref{8vertex} $(I)$. Then, we have the following possibilities:
\begin{itemize}
\item  If the interior of $e_ie_je_k$  does not contribute the same four vertices of $\partial D_i$ to $\partial D_j$, then the interior of $e_je_le_m$ also does not contribute the four vertices to the $\partial D_j$. If the interiors of $e_je_ie_l$(resp. $e_je_ke_l)$, and $e_je_ke_m$ (resp., $e_je_ie_m)$ contributes the $8$ vertices to the $\partial D_j$, then the $3$-simplex $e_ie_je_ke_l$ (resp., $e_ie_je_ke_m$) contribute at least $10$ vertices to the $V(X)$. This is a contradiction. 

    \item If the interior of $e_ie_je_k$ contribute the same four vertices of $\partial D_i$ to $\partial D_j$, then the interior of $e_je_le_m$ contribute the four vertices to the $\partial D_j$. If  the vertex set of $\partial D_k$ is of type Corollary \ref{8vertex} $(II)$ then we can follow the same steps of Subcase (i) above. Therefore, we consider that  the vertex set of $\partial D_k$ is of type $(I)$ in Corollary \ref{8vertex}. Since the $3$-simplex $e_ie_je_le_m$ contribute at least $8$ vertices to the $V(X)$, $e_ie_je_k$ must contribute the same four vertices of $\partial D_i$ to $\partial D_k$. Therefore, the interior of $e_ke_le_m$ contribute the other four vertices to the $V(\partial D_k)$. Thus, the $16$ vertices of $V(X)$ comes from the interiors of faces $e_ie_je_k,e_je_le_m,e_ke_le_m,e_ie_l,$ and $e_ie_m$.   
But then the possible faces whose interiors may contribute $8$ vertices to the $\partial D_l$ are $e_je_le_m, e_ke_le_m, e_ie_l $. Observe that $V(\partial D_l)$ cannot be of any type given in Corollary \ref{8vertex}.
\end{itemize}

Hence, form the all above discussions in cases $(1)$ to $(4)$, we get only one possibility where the $16$ vertices of $X$ comes from the interiors of simplices $e_ie_je_k$, $e_ie_l$, $e_ie_m$, $e_ke_le_m$, $e_je_l$, and  $e_je_m$. 

Without loss of generality, we replace $i $ with $1$, $k$ with $0$, $j$ with $2$, $l$ with $3$ and $m$ with $4$. 
Then $\partial D_0$ has an equivariant triangulation as in  Corollary \ref{8vertex} $(I)$, and there are $8$-vertices, say, $\{u_1, \ldots, u_8\}$ such that $\pi(\{u_1, u_2, u_3, u_4\})$ is a point in the interior of triangle, say $e_0e_1e_2$, of $\Delta^4$, and $\pi(\{u_5, u_6, u_7, u_8\})$ is a point in the interior of a triangle, say $e_0e_3e_4$, of $\Delta^4$. Then, the 8-vertex equivariant triangulation of $\partial D_1$, $\partial D_2$, $\partial D_3$ and  $\partial D_4$ are as in Corollary \ref{8vertex} $(II)$. Under this situation, we will examine the induced triangulation on $\Delta^4$ (see Theorem \ref{thm2}). Figure \ref{fig_trin_D4} gives some faces of this triangulation on $\Delta^4$.
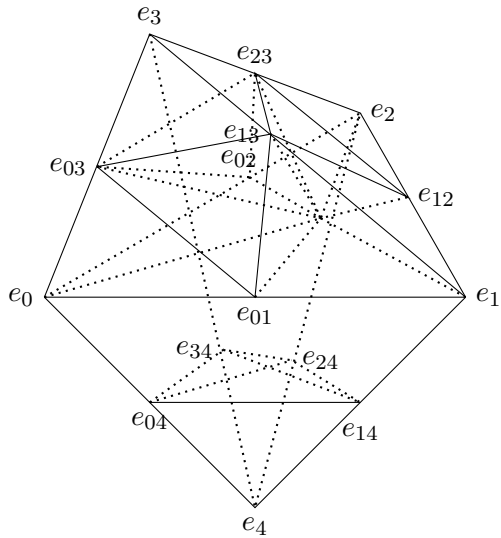
\begin{figure}
\begin{tikzpicture}[scale=0.7]
    \draw[] (0,0)--(8,0);
   \draw[] (0,0)--(4,-4)--(8,0);  
 \draw[] (0,0)--(2,5)--(8,0);
 \draw[] (2,5)--(6,3.5)--(8,0);
 \draw[thick, dotted] (2, 5)--(4, -4)--(6, 3.5)--(0,0);
  \draw[] (1,2.48)--(4,0)--(4.3,3.1)--(1,2.48);
    \draw[thick, dotted] (1,2.48)--(5.25, 1.5)--(4, 4.26);
   \draw[thick, dotted] (1,2.48)--(4, 4.26)--(3.9, 2.28)--(1,2.48);
    \draw[thick, dotted] (3.9, 2.28)--(5.25, 1.5);
     \draw[thick, dotted] (0, 0)--(5.25, 1.5)--(4.3,3.1);
   \draw[thick, dotted] (8, 0)--(5.25, 1.5)--(6.9, 1.9);
   \draw[thick, dotted] (4, 0)--(5.25, 1.5)--(6, 3.5);
   
    \node[left] at (0,0) {$e_0$};
   \node[below] at (4, -4) {$e_4$};
       \node[right] at (8,0) {$e_1$};
   \node[right] at (6, 3.5) {$e_2$};
    \node[above] at (2,5) {$e_3$};
  \node[right] at (6.9, 1.9) {$e_{12}$};
     \node[above] at (4, 4.26) {$e_{23}$};
       \node[below] at (6,-2.2) {$e_{14}$};
   \node[left] at (3.4,-1) {$e_{34}$};
     \node[left] at (4.3,3.1) {$e_{13}$}; 
 \node[right] at (4.7,-1.2) {$e_{24}$};
     \node[left] at (1,2.48) {$e_{03}$};
    \node[below] at (4,0) {$e_{01}$};
  \node[above] at (3.7, 2.28) {$e_{02}$};
   \node[below] at (2,-2) {$e_{04}$};
   \draw[] (2,-2)--(6,-2); 
    \draw[thick, dotted] (2,-2)--(3.4,-1)--(4.7,-1.2)--(6,-2); 
     \draw[] (4.3, 3.1)--(4.3,3.1)--(6.9, 1.9)--(4, 4.26)--(4.3, 3.1); 
        \draw[thick, dotted] (2,-2)--(4.7,-1.2); 
   \draw[thick, dotted] (3.4,-1)--(6,-2); 
 \end{tikzpicture}
\caption{A part of a triangulation on $\Delta^4$.}
\label{fig_trin_D4}
 \end{figure}
 Let $e_{ij}$ (resp., $e_{ijk}$) denote the barycenter of the simplex  $e_ie_j$ (resp., $e_ie_je_k$) in $\D^4$ for $i,j,k\in \{0,\ldots,4\}$. 
The induced triangulation on $F_0=\mathrm{Ind}_{\D^4}(\{e_1,e_2,e_3,e_4\}) $ contains the faces as in Figure \ref{fig_Tring_F0}. 
 \begin{figure}
\begin{tikzpicture}[scale=0.7]
 \draw[] (2,5)--(6,3.5)--(8,0)--(2,5)--(4,-4)--(8,0);
 \draw[thick, dotted] (2, 5)--(4, -4)--(6, 3.5);
 \draw[] (4.3, 3.1)--(6,-2)--(3.3,-1)--(4.3,3.1)--(6.9, 1.9)--(4, 4.26)--(4.3, 3.1);
 \draw[thick, dotted] (4.8, -1)--(6.9, 1.9)--(6,-2)--(4.8, -1)--(3.3,-1)--(4, 4.26)--(4.8, -1);

 \node[right] at (6.9, 1.9) {$e_{12}$};
   \node[above] at (6, 3.5) {$e_2$};
    \node[above] at (4, 4.26) {$e_{23}$};
  \node[below] at (4,-4) {$e_4$};
    \node[below] at (8,0) {$e_1$};
      \node[below] at (6.2,-2) {$e_{14}$};
   \node[left] at (2,5) {$e_{3}$};
   \node[left] at (3.25,-1) {$e_{34}$};
     \node[left] at (4.3,3.1) {$e_{13}$}; 
 \node[below] at (5.6, 1.7) {$e_{24}$};
    
\end{tikzpicture}
\caption{A triangulation on $F_0$.}
\label{fig_Tring_F0}
 \end{figure}
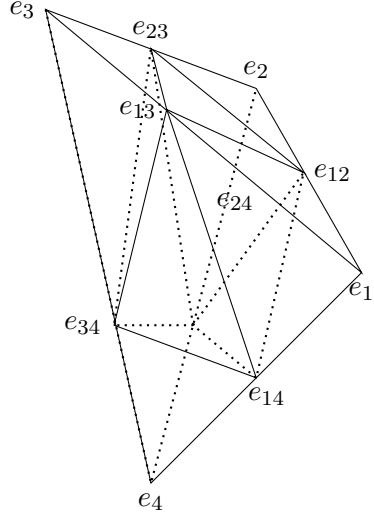
Suppose that there is an edge between $e_{12}$ and $e_{34}$, say $e_{12}e_{34}$. Then $(-1, -1, 0,0)(e_{12}e_{34})=(-e_{12})e_{34}$ is also an edge in $X$. This is a contradiction, since $e_{12}=-e_{12}$ in $\mathbb{RP}^4$. By similar arguments, there cannot be an edge between $e_{14}$ and $e_{23}$. 

Now, to obtain the induced triangulation of $\Delta^4$, we need to have edges either from $e_{012}$ to each of $e_{03}, e_{23}, e_{24}$, or from $e_{134}$ to each of vertices $e_{02}, e_{13}, e_{23}$. However, similarly as above, there cannot be an edge between the vertices  $e_{23}$, $e_{012}$, $e_{34}$, $e_{034}$, and $e_{23}$. Observe that the tetrahedron $e_{13}e_{14}e_{012}e_{034}$ is a simplex in $X$. So, it is a face of two $4$-simplices. The other vertex of these two simplices should come from $\{e_{24}, e_{34}, e_{12}, e_{23}\}$. Now, $e_{13}e_{14}e_{012}e_{034}e_{24}$ cannot be a $4$-simplex, since $e_{13}e_{24}$ is not an edge in $X$. Similarly, $e_{13}e_{14}e_{012}e_{034}e_{34}$, $e_{13}e_{14}e_{012}e_{034}e_{12}$, and $e_{13}e_{14}e_{012}e_{034}e_{23}$ cannot be a $4$-simplex  $\Delta^4$. Therefore, there is no $16$-vertex $\ZZ_2^4$-equivariant triangulation of $\mathbb{R}P^4$.

If $Y$ is a $17$-vertex equivariant triangulation of $\mathbb{R}P^4$, then one vertex of $Y$ has to be a fixed vertex (see Lemma \ref{lem1}), say $e_0$. The situation of the vertices in $V(Y)\setminus\{e_0\}$ will be as in the case for $X$ above. Thus, by similar arguments,  $Y$ cannot exist.  
\end{proof}

Theorem \ref{thm7} raises the following question.
\begin{question}
Is there a $\ZZ_2^4$-equivariant triangulation of $\RP^4$ with $18$ vertices?  
\end{question}

Proposition \ref{equivariant} gives some equivariant triangulation of $\mathbb{R}P^n$. However, the number of verteces is very large. Therefore, it is natural to ask the following.
\begin{question}
Fix $n\geq 4$. What is the minimal number of vertices required for a $\ZZ_2^n$-equivariant triangulation of $\RP^n$?
\end{question}

{\bf Acknowledgment.} The first author would like to thank Subhajit Ghosh for his support through the Inspire Faculty Fellow Research Grant (IFA-23-MA 198) and IIT Madras for the  institute postdoctoral fellowship.
The authors thank Basudeb Datta, Mainak Poddar, and Sonia Balagopalan
for their helpful comments. 
The second author thanks Korea Advanced Institute of Science and  Technology, Pacific Institute
for the Mathematical Sciences and the University of Regina for financial support.

\end{document}